\documentclass[12pt,leqno,A4]{amsart}

\usepackage{graphicx, mathtools, hyperref, verbatim, xypic}
\usepackage{tikz-cd}

\usepackage{amssymb}
\usepackage{amsthm}
\usepackage{amsmath}
\usepackage{amscd}
\usepackage[mathscr]{eucal}

\newtheorem{theorem}{Theorem}[section]
\newtheorem{proposition}[theorem]{Proposition}
\newtheorem{lemma}[theorem]{Lemma}
\newtheorem{corollary}[theorem]{Corollary}

\theoremstyle{definition}
\newtheorem{definition}[theorem]{Definition}

\theoremstyle{remark}
\newtheorem{remark}[theorem]{Remark}

\numberwithin{equation}{section}

\newcommand{\R}{\mathbb{R}}

\newcommand{\Z}{\mathbb{Z}}

\renewcommand{\H}{\mathbb{H}}

\newcommand{\SL}{\operatorname{SL}}

\newcommand{\Mp}{\operatorname{Mp}}

\newcommand{\SO}{\operatorname{SO}}

\begin{document}

\title[The Shimura-Shintani Correspondence via singular theta lifts]
{The Shimura-Shintani Correspondence via singular theta lifts and currents}

\author{Jonathan Crawford}
\address{172, Gloucester Road, 
Cheltenham,
Gloucestershire, GL51 8NR,
United Kingdom}
\email{j.k.crawford@live.co.uk}

\author{Jens Funke}
\address{Department of Mathematical Sciences,
Durham University,
South Road,
Durham, DH1 3LE,
United Kingdom}
\email{jens.funke@durham.ac.uk}


\date{\today.}



\begin{abstract}
We describe the construction and properties of a singular theta lift for the orthogonal group $\SO(2,1)$. We obtain locally harmonic Maass forms in the sense of \cite{BKK} with singular sets along geodesics in the upper half plane. We consider 
these forms as currents and derive properties of the Shimura-Shintani correspondence. This work provides extensions of the theta lifts considered in \cite{Hoevel} and \cite{BKV}.
\end{abstract}

\maketitle

\section{Introduction and Statement of Results}
\subsection{Introduction}
The modern theory of modular forms of half-integral weight was initiated by Shimura \cite{Shimura}, where he constructed a family of maps from half-integral weight $k+1/2$ cusp forms to even weight $2k$ holomorphic modular forms. Shortly afterwards the Shimura correspondence was realised as a theta lift by Niwa \cite{Niwa} while Shintani \cite{Shintani} described the adjoint map also as a theta lift. Much later Borcherds \cite{Borcherds} introduced a singular theta lift of modular forms. This encompassed the Shimura lift, as well as many other examples. Borcherds used a regularisation of the theta integral to enlarge the space of inputs from cusp forms to weakly holomorphic modular forms. The Borcherds lift also gave rise to remarkable product expansions for automorphic forms associated to orthogonal groups. The extension of the constructions from \cite{Borcherds} by Bruinier and the second author \cite{B-Habil,BF-Duke} led to the introduction of harmonic weak Maass forms. This coincided with the work of Zwegers \cite{Zwegers} who showed that Ramanujan's famous mock theta functions were holomorphic parts of harmonic weak Maass forms. Harmonic weak Maass forms have had many recent applications ranging from combinatorial number theory, arithmetic geometry, to mathematical physics. 

Another class of automorphic objects, locally harmonic weak Maass forms, was introduced by Bringmann, Kane, and Kohnen \cite{BKK}. These forms are similar to harmonic weak Maass forms but may also exhibit singularities. 
H{\"o}vel \cite{Hoevel} described a singular theta lift of weight $1/2$ harmonic weak Maass forms to weight $0$ locally harmonic Maass forms which he then linked to the Shimura lift in the case $k=1$. Bringmann, Kane, and Viazovska \cite{BKV} considered for $k$ even the lift of certain scalar-valued Poincar\'{e} series of full level and weight $3/2-k$. The lift in the opposite direction was considered by Alfes, Griffin, Ono, and Rolen \cite{AGOR} for $k=0$ and by Alfes-Neumann and Schwagenscheidt  \cite{AS} for general $k$. 

This paper further develops these ideas. For $k \geq 1$, we consider a regularised theta lift of harmonic weak Maass forms of weight $3/2-k$ by integrating against a suitable theta kernel associated to the dual pair $\mathrm{O}(2,1)$  and $\mathrm{SL}_2(\mathbb{R})$ (essentially the same one considered in \cite{BKV,BKK}). Using this, we obtain locally harmonic Maass forms of weight $2-2k$ with singularities along a collection of geodesics. We explicitly compute the Fourier expansion and give the wall crossing formulas when moving from one Weyl chamber to another. We link our lift to the Shimura lift which then allows us to re-derive properties of the Shimura lift. Finally, we also consider locally harmonic weak Maass forms and the singular lift as distributions which we believe is the proper setting for such forms. In the present situation this enables us to relate our work to the Shintani lift.

\subsection{Statement of Results}

Throughout, for $N \in \mathbb{N}$, we consider the quadratic space $(V,Q)$ given by 
\[  
V \coloneqq \left\{ \lambda \in \mathrm{M}_2(\mathbb{Q}) \mid \mathrm{tr}(\lambda) = 0 \right\} 
\] 
with quadratic form $Q(\lambda) \coloneqq - N \det(\lambda)$. Note $(V,Q)$ has signature $(2,1)$. We consider the even lattice
\[ L \coloneqq \left\{\left(
\begin{smallmatrix}
	b & -a/N
\\	c & -b
\end{smallmatrix}\right)
\bigg| \ a,b,c \in \mathbb{Z}\right\},
\]
which has level $4N$, and we have $L'/L \cong \mathbb{Z}/2N\mathbb{Z}$ where $L'$ denotes the dual lattice.

We let $\rho_{L}$ be the associated Weil representation of the inverse image of $\SL_2(\Z)$ under the covering map of the metaplectic group $\Mp_2(\R)$ to $\mathrm{SL}_{2}(\mathbb{R})$. We let $\kappa \in \frac{1}{2}\mathbb{Z}$. We write $M_{\kappa, \rho_{L}}$ ($S_{\kappa, {\rho}_{L}}$) for the space of vector-valued holomorphic modular (cusp) forms of weight $\kappa$ for $\rho_L$.

We let $H_{\kappa, \rho_{L}}$ be the space of {harmonic weak Maass forms} (see  \cite{BF-Duke}) which consists of real analytic forms of weight $\kappa$ on the upper half plane $\mathbb{H}$ of exponential growth which are mapped under the operator $\xi_{\kappa} = 2iv^{\kappa}\overline{\frac{\partial}{\partial\overline{\tau}}}$ ($\tau=u+iv \in \mathbb{H}$) to the space of cusp forms $S_{2-\kappa, \overline{\rho}_{L}}$ of weight $2-\kappa$. The Fourier expansion of a form 
$f \in H_{\kappa, \rho_{L}}$ is given by 
\[  
f(\tau)= \sum_{h \in L'/L}\sum_{ n \gg - \infty }c^{+}(n,h)e(n \tau)\mathfrak{e}_{h} + \sum_{h \in L'/L}\sum_{n < 0 }c^{-}(n,h)\Gamma(1-\kappa,4\pi |n|v)e(n \tau)\mathfrak{e}_{h}.
\]
Here $\Gamma( \cdot, \cdot)$ denotes the incomplete gamma function, and $\mathfrak{e}_{h}$ is the standard basis element of $\mathbb{C}[L'/L]$ corresponding to $h \in L'/L$. 
We also consider the space of (scalar-valued) \emph{locally harmonic weak Maass forms} $LH_{\kappa}$ ($\kappa \in 2\mathbb{Z}$), see \cite{BKK}. These forms mirror $H_{\kappa}$ but are only harmonic in connected components, away from an  exceptional measure zero set. We fix $k \in \mathbb{Z}, k \geq 1$, and let $\Delta \in \mathbb{Z}$ be a fundamental discriminant giving rise to a genus character. We also let $r \in \mathbb{Z}$ such that $\Delta \equiv r^{2} \pmod{4N}$. With this data we consider two closely related kernel functions $\Theta_{\Delta,r,k}(\tau,z)$ and $\Theta_{\Delta,r,k}^{*}(\tau,z)$ which underlie the singular theta lift and the Shimura lift with the following transformation properties:
\begin{enumerate}
 \itemsep-0.2em  \item $\Theta_{\Delta,r,k}(\tau,z)$ has weight $k-3/2$ in $\tau$ for $\tilde{\Gamma}$. It has weight $2-2k$ in $z$ for $\Gamma_{0}(N)$.
\item $\Theta_{\Delta,r,k}^{*}(\tau,z)$ has weight $k+1/2$ in $\tau$ for $\tilde{\Gamma}$. It has weight $2k$ in $-\overline{z}$ for $\Gamma_{0}(N)$. 
\item $\xi_{k+1/2, \tau}( \Theta_{\Delta,r,k}^{*}(\tau,z)) = -\frac{1}{2}\xi_{2-2k,z}\left(v^{k-3/2}\Theta_{\Delta,r,k}(\tau,z)\right)$.
\end{enumerate} 
These properties are of crucial importance for the geometric aspects of the theta correspondence and singular theta lifts for arbitrary signature, see \cite{KM90} and \cite{BF-Duke} (in our case of $\mathrm{O}(2,1)$ it is the case $k=1$). For $k >1$, it can be also found in the setting of differential forms with values in local systems in \cite{FM-coeff}. 

For $f \in H_{3/2-k, \overline{\rho}}$ we then define the \emph{singular theta lift} as the regularised Petersson scalar product 
\[ 
\Phi_{\Delta,r,k}(z,f) \coloneqq  \int_{ \mathcal{F}}^{\mathrm{reg}} \left\langle f(\tau), \overline{\Theta_{\Delta,r,k}(\tau,z)}\right\rangle \frac{dudv}{v^2}.
\]
The asymptotic behaviour of $f$ means this integral diverges in general, hence needs to be generalized along the lines introduced by Harvey and Moore and Borcherds \cite{HM,Borcherds}. The regularised integral $\Phi_{\Delta,r,k}(z,f)$ then converges pointwise for any $z \in \mathbb{H}$ with singularities along $Z'_{\Delta,r}(f)$, a locally finite collection of geodesics in the upper half plane associated to the principal part of $f$. These geodesics divide $\mathbb{H}$ into connected components called Weyl chambers. Each component of $Z'_{\Delta,r}(f)$ arises from a geodesic associated to a rational vector of positive length. We show 

\begin{theorem}[Theorem \ref{lift is a local maasd form theorem}]The singular theta lift $\Phi_{\Delta,r,k}(z,f)$ defines a locally harmonic Maass form of weight $2-2k$ for $\Gamma_{0}(N)$ with exceptional set $Z'_{\Delta,r}(f)$:
\[ 
\Phi_{\Delta,r,k}: H_{3/2-k, \overline{\rho}} \longrightarrow LH_{2-2k}(\Gamma_{0}(N)).
\]
\end{theorem}
We find explicit \emph{wall crossing formulas} (Theorem \ref{wall cross theorem}) which tell us the nature of the singularities along $Z'_{\Delta,r}(f)$.
We also compute the Fourier expansion of $ \Phi_{k}(z,f)$. A simplified version omitting constants is as follows. 

\begin{theorem}[Theorem \ref{The Fourier Expansion}]\label{intro fourier expansion}
Assume $\Delta=1,r=1,k \geq 2$ even. Then for $y=Im(z)$ sufficiently large we have 
\begin{multline*} 
\Phi_{k}(z,f) {} = {}  c^{+}(0,0)\zeta(k) + \sum_{m \geq 1}c^{+}\left(-\tfrac{m^{2}}{4N},\tfrac{m}{2N}\right) B_{k}\left(mz+\lfloor mx \rfloor \right)
\\  + \sum_{m \geq 1}\sum_{n \geq 1}c^{-}\left(-\tfrac{m^{2}}{4N},\tfrac{m}{2N}\right)\left[e(nmz)+\Gamma(2k-1,4 \pi n m y) e(-nmz)\right]n^{-k}.
\end{multline*}
Here $B_{k}(x)$ is the $k$-th Bernoulli polynomial and $\zeta(s)$ denotes the Riemann zeta function. 
\end{theorem}
Note that vertical half-line singularities of $\Phi_{k}(z,f)$ are encompassed by the periodic Bernoulli polynomials in this expansion.

Let $g= \sum_{h \in L'/L}\sum_{n>0}a(n,h)e(n \tau) \mathfrak{e}_{h} \in S_{k+1/2, \rho}$ be a cusp form. Then in our setting the \emph{Shimura lift} is given by 
\begin{align*} \Phi_{\Delta,r,k}^{*}(z,g) \coloneqq \int_{\tau \in \mathcal{F}} \left\langle g(\tau), \Theta_{\Delta,r,k}^{*}(\tau,z)\right\rangle v^{k+1/2}\frac{dudv}{v^2}.
\end{align*}
We have the following key theorem which links the two lifts.
\begin{theorem}[Theorem \ref{Link}]\label{link intro}
For $f \in H_{3/2-k,\overline{\rho}}$ and $z \in \mathbb{H} \backslash Z'_{\Delta,r}(f)$ then
\[ \Phi_{\Delta,r,k}^{*}(z,\xi_{3/2-k,\tau}(f)) = \frac{1}{2}\xi_{2-2k,z}(\Phi_{\Delta,r,k}(z,f)). \]
\end{theorem} 
This link allows us to give new proofs of properties of the Shimura lift. Firstly, we can find its Fourier expansion by applying $\xi_{2-2k}$ to Theorem \ref{intro fourier expansion}. For example, if $\Delta=1,r=1, k \geq 2$ then 
\begin{align*} 
\Phi_{k}^{*}\left(z,g \right) =  \sum_{m \geq 1}\sum_{\substack{d \geq 1 \\ d|m }}d^{k-1} a\left(\tfrac{m^{2}}{4Nd^{2}},\tfrac{m}{2Nd} \right)  e\left(m z \right).
\end{align*}
Considering the action under  Atkin-Lehner involutions we then also recover that the Shimura lift maps cusp forms to {cusp forms} if $k \geq 2$ or $k=1, \Delta \neq 1$. In summary, we have the following commutative diagram:
\[      
   \begin{tikzcd}
   H_{3/2-k,\overline{\rho}}  \arrow[r,  "\qquad \Phi_k \qquad" ] \arrow[d,"\xi_{3/2-k}"']& LH_{2-2k}(\Gamma_{0}(N)) \arrow[d, "\xi_{2-2k}" ] \\ S_{k+1/2, \rho} \arrow[r, "\Phi^*_k" ] & S_{2k}(\Gamma_{0}(N)).
   \end{tikzcd}
  \]   
  
 \begin{remark}
We can also apply the iterated holomorphic differential $\mathcal{D}^{2k-1}= \left(\tfrac1{2\pi i} \tfrac{\partial}{\partial z}\right)^{2k-1}$  to the lift $\Phi_{\Delta,r,k}$ (which by Bol's Lemma is the same as applying the Maass raising operator $2k-1$ times). From the Fourier expansion it is easy to see that one again obtains up to a constant multiple $\Phi_{\Delta,r,k}^{*}(z,\xi_{3/2-k,\tau}(f))$, the Shimura lift of the shadow of $f$. This feature of locally harmonic Maass forms (which differs from the behavior of usual harmonic Maass forms, see \cite{BOR}) was first observed in \cite{BKK}.
 \end{remark} 
   
We also consider the singular theta lift as a \emph{distribution}. We fix a space of {test functions} $g(z) \in A_{\kappa}^{rd}(\Gamma_{0}(N))$ of smooth weight $\kappa$ forms of rapid decay. Let $h(z) \in LH_{\kappa}(\Gamma_{0}(N))$. Then we define the distribution associated to $h(z)$ as
\[ \left[h \right] (g) \coloneqq (g,h)_{\kappa} = \int_{\Gamma_{0}(N) \backslash \mathbb{H}} g(z)\overline{h(z)} y^{\kappa} \frac{dx dy}{y^{2}},
\]
and define a distributional derivative by $\xi_{\kappa}[h](g) \coloneqq  -(h, \xi_{2-\kappa}(g))_{\kappa}$ for $g \in A_{2-\kappa}^{rd}(\Gamma_{0}(N))$. We then have the following distributional version of Theorem \ref{link intro}. 

\begin{theorem}[Theorem \ref{current equation for local}] 
Let $g \in A_{2k}^{rd}(\Gamma_{0}(N))$. Then
\begin{align} \label{intro dist link}
\xi_{2k}\left[ \Phi_{\Delta,r,k}(z,f) \right](g)  & = \left[ \xi_{2-2k}(\Phi_{\Delta,r,k}(z,f)) \right](g) -  \int_{Z_{\Delta,r}(f)} g(z) q_{z}(\lambda)^{k-1} dz \\
&=2\left[ \Phi^{*}_{\Delta,r,k}(z,\xi_{3/2-k}(f)) \right](g) -  \int_{Z_{\Delta,r}(f)} g(z) q_{z}(\lambda)^{k-1} dz. \notag
\end{align}
\end{theorem}
Here ${Z_{\Delta,r}(f)}$ is the image of ${Z'_{\Delta,r}(f)}$ in the modular curve $Y_0(N)$. 
As an application we consider $g \in S_{2k}(\Gamma_{0}(N))$. Then the left hand side of \eqref{intro dist link} vanishes, and we recover the Fourier expansion of the Shintani lift of $g$.

\medskip

The results in this work are based on the first author's Ph.D. thesis \cite{Crawford} at the University of Durham under the supervision of the second author. The second author also would like to thank the Max Planck Institute for Mathematics in Bonn for its hospitality and stimulating working conditions during his stay in 2020. The authors further thank Claudia Alfes-Neumann and Markus Schwagenscheidt for their helpful comments.


\section{Preliminaries}

 \subsection{The Weil Representation}

For $g = \begin{psmallmatrix} a& b \\ c & d\end{psmallmatrix}$ we write $j(g,\tau) = c \tau +d$ as usual. We let $\mathrm{Mp}_{2}(\mathbb{R})$ be the (unique) non-trivial two-fold cover of $\mathrm{SL}_{2}(\mathbb{R})$. We realize $\mathrm{Mp}_{2}(\mathbb{R})$ as the pairs $(g, \phi_g)$ with $g \in \mathrm{SL}_{2}(\mathbb{R})$ and  $\phi_{g} : \mathbb{H} \rightarrow \mathbb{C}$ is a holomorphic function such that $\phi_{g}^2(\tau) = j(g, \tau)$. We define $\tilde{\Gamma}$ to be the inverse image of $\Gamma \coloneqq \mathrm{SL}_{2}(\mathbb{Z})$ under the covering map $\mathrm{Mp}_2(\mathbb{R}) \rightarrow \mathrm{SL}_2(\mathbb{R})$. $\tilde{\Gamma}$ is generated by $T \coloneqq \left( \begin{psmallmatrix} 1 & 1 \\ 0 & 1 \end{psmallmatrix} , 1 \right)$ and $S \coloneqq \left( \begin{psmallmatrix} 0 & -1 \\ 1 & 0 \end{psmallmatrix} , \sqrt{\tau} \right)$ .
We also define the group $\tilde{\Gamma}_{\infty} = \{ \left(\begin{psmallmatrix}
	1 & n
	\\ 0 & 1
\end{psmallmatrix},1 \right); n \in \mathbb{Z} \} \subset \tilde{\Gamma}.$

Let $(V,Q)$ be a rational non-degenerate quadratic space of signature $(b^{+},b^{-})$ with bilinear form $(\cdot,\cdot)$. Let $L$ be an even lattice ($Q(\lambda) = \tfrac{1}{2}(\lambda,\lambda) \in \mathbb{Z}$ for all $\lambda \in L$) and $L'$ the dual lattice. 
We equip the discriminant group $L'/L$ with the induced quadratic form $\overline{Q}: L'/L \rightarrow \mathbb{Q}/\mathbb{Z}$. We write $L^{-} \subset V^{-}$ for the lattice $L \subset V$ with bilinear form $-(,)$. We let $\mathbb{C}[L'/L]$  be $\mathbb{C}$-group algebra consisting of formal linear combinations $\sum_{h \in L'/L} \lambda_{h}\mathfrak{e}_{h}$ where $\lambda_{h} \in \mathbb{C}$ and $\mathfrak{e}_{h}$ is the standard basis element corresponding to $h \in L'/L$. 
We define a Hermitian scalar product on $\mathbb{C}[L'/L]$ by letting  $\left\langle \mathfrak{e}_{h}, \mathfrak{e}_{h'} \right\rangle \coloneqq \delta_{h,h'}$
and extend this to $\mathbb{C}[L'/L]$ by sesquilinearity. For a function $f : \mathbb{H} \rightarrow \mathbb{C}[L'/L]$ we denote the components by $f_{h}$ so that $f = \sum_{h \in L'/L} f_{h} \mathfrak{e}_{h}$.

The (finite) Weil representation $\rho_L$ acts on  $\mathbb{C}[L'/L]$, see eg \cite[Section~1]{Shintani}, \cite[Section~4]{Borcherds}. On the generators it is given by 
\begin{align*} \rho_{L}(T)(\mathfrak{e}_{h}) \coloneqq e(Q(h))\mathfrak{e}_{h},
\qquad \qquad  \rho_{L}(S)(\mathfrak{e}_{h}) \coloneqq \frac{e((b^{-} - b^{+})/8)}{\sqrt{|L'/L|}}\sum_{h' \in L'/L}e(-(h,h'))\mathfrak{e}_{h'}. 
\end{align*}
We note $\rho_{L^{-}} = \overline{\rho}_{L}$.

\subsection{Vector-Valued Forms}
Let $\kappa \in \frac{1}{2}\mathbb{Z}$. We write $A_{\kappa,\rho_{L}}$ for the space of smooth vector-valued automorphic forms of weight $\kappa$ with respect to $\rho_L$. These are smooth functions $f: \mathbb{H} \rightarrow \mathbb{C}[L'/L]$ such that 
\[  \left( f |_{\kappa, \rho_{L}} (\gamma, \phi_{\gamma})\right)(\tau) \coloneqq \phi_{\gamma}(\tau)^{-2\kappa}\rho_{L}(\gamma, \phi_{\gamma})^{-1}f(\gamma \tau)= f(\tau) 
\]
for all $(\gamma, \phi_{\gamma}) \in \tilde{\Gamma}$. We denote $M^{!}_{\kappa,\rho_{L}}, M_{\kappa,\rho_{L}}$ and $S_{\kappa,\rho_{L}}$ for the subspaces of (vector-valued) weight $\kappa$ weakly holomorphic modular, holomorphic modular and cusp forms respectively. If $f \in A_{\kappa,\rho_{L}}$ satisfies
\begin{enumerate}
\item $\Delta_{\kappa} f = 0$,
\item there is a $C>0$ such that $f(\tau) = O(e^{Cv})$ as $v \rightarrow \infty$,
\end{enumerate}
then we call $f$ a harmonic weak Maass form, see also \cite[Section~3]{BF-Duke}. We denote the space of these as $\mathcal{H}_{\kappa,\rho_{L}}$. Here $ \Delta_{\kappa} = \Delta_{\kappa,\tau} \coloneqq -v^{2}\left(\frac{\partial^{2}}{\partial u^{2}} +  \frac{\partial^{2}}{\partial v^{2}}\right) + i\kappa v \left(\frac{\partial}{\partial u} + i \frac{\partial}{\partial v}\right)$ is the weight $\kappa$ hyperbolic Laplacian operator. We also have the subspace $H_{\kappa,\rho_{L}} \subset \mathcal{H}_{\kappa,\rho_{L}}$ where we request that there exists a Fourier polynomial 
\[ 
P_{f} \coloneqq \sum_{h \in L'/L}\sum_{\substack{n \in \mathbb{Z} + Q(h) \\ - n_0 \leq n\leq0}}c^{+}(n,h)e(n\tau)\mathfrak{e}_{h}
 \] so that $f(\tau) - P_{f}(\tau) = O(e^{-\epsilon v}) $ as $v \rightarrow \infty$. We call $P_{f}$ the principal part of $f$. The principal part $P_{f}$ is uniquely determined by $f$. 
Any $f \in H_{\kappa,\rho_{L}}$ has a unique decomposition $f=f_+ + f_-$ with
\begin{subequations}\label{Decomposition}
\begin{align}
f^{+} &\coloneqq \sum_{h \in L'/L}\sum_{\substack{n \in \mathbb{Z} + Q(h) \\ n \geq -n_0}}c^{+}(n,h)e(n \tau)\mathfrak{e}_{h}, \label{c+ decomp}
\\ f^{-} &\coloneqq \sum_{h \in L'/L}\sum_{\substack{n \in \mathbb{Z} + Q(h) \\ n < 0 }}c^{-}(n,h)\Gamma(1-\kappa,4\pi |n|v)e(n \tau)\mathfrak{e}_{h},
\end{align}
\end{subequations}
and for $\kappa \geq 2$, the $f^{-}$ part vanishes. Here $\Gamma(a,s)$ is the incomplete $\Gamma$-function. We call $f^{+}$ the holomorphic part and $f^{-}$ the non-holomorphic part. We note that for $f \in H_{\kappa,\rho_{L}}$ then $c^{\pm}(n,h) = (-1)^{\kappa+\frac{b^{-}-b^{+}}{2}}c^{\pm}(n,-h)$. We let $
R_{\kappa} {} \coloneqq {}  2i \tfrac{\partial}{\partial \tau}+\kappa v^{-1} , \quad L_{\kappa} {} \coloneqq {}   -2iv^{2}\tfrac{\partial}{\partial \overline{\tau}}$ be the standard Maass raising and lowering operators, and we define the anti-linear differential operator 
\[
\xi_{\kappa}(f)(\tau) \coloneqq v^{\kappa-2}\overline{L_{\kappa}f(\tau)} = R_{-\kappa}v^{\kappa}\overline {f(\tau)}. 
\]
Then $-\Delta_{\kappa}=L_{\kappa+2}R_{\kappa} + \kappa = R_{2-\kappa}L_{\kappa} = \xi_{2-\kappa} \xi_{\kappa}$. By \cite{BF-Duke} the assignment $f \mapsto \xi_{\kappa}(f)$ defines surjective maps \[ \xi_{\kappa} :\mathcal{H}_{\kappa, \rho_{L}} \rightarrow M_{2-\kappa, \overline{\rho}_{L}}^{!}, \qquad 	\xi_{\kappa} : H_{\kappa, \rho_{L}} \rightarrow S_{2-\kappa, \overline{\rho}_{L}},
\] 
and we have
\begin{equation}\label{xi of f}\xi_{\kappa}(f)(\tau) = - \sum_{h \in L'/L}\sum_{\substack{n \in \mathbb{Z} - Q(h) \\ n > 0 }}(4 \pi n)^{1-\kappa}\overline {c^{-}(-n,h)}e(n \tau)\mathfrak{e}_{h}.
\end{equation}
For $f,g \in A_{\kappa,\rho_{L}}$ we let
\[(f,g)_{\kappa,\rho_{L}}^{\mathrm{reg}} \coloneqq \int_{\mathcal{F}}^{\mathrm{reg}} \left\langle f(\tau),g(\tau) \right\rangle v^{\kappa} \frac{du dv}{v^2} \coloneqq \lim_{t \rightarrow \infty}\int_{\tau \in \mathcal{F}_{t}} \left\langle f(\tau),g(\tau) \right\rangle v^{\kappa} \frac{du dv}{v^2}
\]
the regularised Petersson scalar product, whenever this limit exists. Here $\mathcal{F}_{t}$ is the truncated fundamental domain $ \mathcal{F}_{t} := \left\{ \tau \in \mathcal{F} \mid \Im(\tau) \leq t \right\}.$

\subsection{Locally Harmonic Maass Forms}

We give a definition adapting \cite{BK,BKK,BKV}. For a not necessarily continuous function $f : \mathbb{H} \rightarrow \mathbb{C}$ and a nowhere dense exceptional set $E \subset \mathbb{H}$, we will denote $f_{W}$ as the restriction of $f$ to a connected component $W \subset \mathbb{H} \backslash E$. For a point $\tau \in \mathbb{H}$ we write $W_{\tau}^{E}$ for the connected components that contain $\tau$ in their closure i.e. $
W_{\tau}^{E}\coloneqq \left\lbrace W \subset \mathbb{H} \backslash E \ \vert \ \tau \in \overline{W}  \right\rbrace$. Finally let
\[ 
\mathcal{A}_{E}(f)(\tau) \coloneqq \frac{1}{\#W_{\tau}^{E}}\sum_{W \in W_{\tau}^{E}} \lim_{\substack{w \in W \\ w \rightarrow \tau}}f(w)
\]
be the average value of $f$ on the connected components in which $\tau$ lies (whenever this limit exists).
\begin{definition}\label{locally harmonic weak maass form}
Let $\kappa \in 2\mathbb{Z}, \kappa \leq 0$, $\Gamma' \subset \Gamma$ be a finite index subgroup and $E$ be a $\Gamma'$-invariant exceptional set $E \subset \mathbb{H}$. We will call a function $f:\mathbb{H} \rightarrow \mathbb{C}$ a (scalar-valued) locally harmonic weak Maass form, of weight $\kappa$ for $\Gamma'$ and $E$, if
\begin{enumerate}
	\item $f|_{\kappa}\gamma = f$ for all $\gamma \in \Gamma'$.
	\item For all $\tau \in \mathbb{H} \backslash E$ there is a neighbourhood $U \subset \mathbb{H}$ of $\tau$ in which $f$ is real analytic and $\Delta_{\kappa}f=0$.
	\item For all $\tau \in \mathbb{H}$ we have that $W_{\tau}^{E}$ is a finite set, the limit defining $\mathcal{A}_{E}(f)(\tau)$ exists, and $f= \mathcal{A}_{E}(f)$.
	\item For any cusp $f$ has polynomial growth. 
	\end{enumerate}
\end{definition}

We will call the real analytic connected components Weyl chambers as in \cite[Section~6]{Borcherds} 
and denote the space of locally harmonic weak Maass forms as $LH_{\kappa}(\Gamma')$.

\subsection{Siegel Theta Functions}
We define the Grassmannian, $\mathrm{Gr}(V(\mathbb{R}))$ as
\[ \mathrm{Gr}(V(\mathbb{R})) \coloneqq \left\{ z \subset V(\mathbb{R}) \mid \dim \ z = b^{-} \ \textrm{and} \ Q|_{z} < 0 \right\}. \]
We have $\mathrm{Gr}(V(\mathbb{R})) \cong \mathrm{SO}^{+}(b^{+},b^{-}) /\mathrm{SO}(b^{+}) \times \mathrm{SO}(b^{-})$. For $z \in \mathrm{Gr}(V(\mathbb{R}))$ and $\lambda \in V(\mathbb{R})$ we have $V(\mathbb{R})=z \oplus z^{\perp}$ and the orthogonal decomposition $\lambda = \lambda_{z} + \lambda_{z^{\perp}}$. We define the majorant $Q_{z}(\lambda) \coloneqq Q(\lambda_{z^{\perp}})-Q(\lambda_{z})$. Following \cite[Section~4]{Borcherds} we let $z \in \mathrm{Gr}(V(\mathbb{R})), h \in L'/L$, $\sigma: V(\mathbb{R}) \rightarrow \mathbb{R}^{b^{+}, b^{-}}$ be an isometry, $p$ be a harmonic homogeneous polynomial on $\mathbb{R}^{b^{+}, b^{-}}$ of degree $(m^{+},m^{-})$ and let $\alpha, \beta \in V(\mathbb{R})$. Then we have a Siegel theta function
\begin{multline*}  \vartheta_{L+h}(\tau, z, \sigma, p, \alpha,\beta) \\ \coloneqq 
  v^{\frac{b^{-}}{2} + m^{-}}\sum_{\lambda \in L + h}p(\sigma(\lambda + \beta)) e\left(Q(\lambda + \beta)u+Q_{z}(\lambda+ \beta)iv - \left( \lambda + \beta/2, \alpha \right) \right),
\end{multline*}
and a $\mathbb{C}[L'/L]$ version $\vec{\vartheta}_{L}(\tau, z, \sigma, p, \alpha, \beta) \coloneqq \sum_{h \in L'/L}\vartheta_{L+h}(\tau, z, \sigma,p, \alpha, \beta)\mathfrak{e}_{h}$. We know (eg \cite{Borcherds}) for any $(\gamma, \phi_{\gamma}) \in \tilde{\Gamma}, \gamma = \left( \begin{smallmatrix} a&b \\ c&d \end{smallmatrix} \right)$ that
\begin{equation}\label{theta transform general} 
\vec{\vartheta_{L}}(\gamma\tau, z, \sigma, p, a\alpha + b\beta, c\alpha + d\beta) = \phi_{\gamma}(\tau)^{b^{+}-b^{-}+2(m^{+}-m^{-})}\rho_{L}(\gamma,\phi_{\gamma})\vec{\vartheta_{L}}(\tau, z, \sigma, p, \alpha, \beta). 
\end{equation}

\section{The Setting}
Discussions of the following can also be found in \cite[Section~2.1]{BIF}. Let $N \in \mathbb{Z}, N>0$ and we let $V := \left\{ \lambda \in M_2(\mathbb{Q}) \ | \ \mathrm{tr}(\lambda) = 0 \right\}$ with $Q(\lambda) = -N \det(\lambda)$ and $(\lambda, \mu) = N \mathrm{tr}(\lambda \mu)$ for $\lambda, \mu \in V$. We fix the lattice
\[ L \coloneqq \left\{
\begin{pmatrix}
	b & -a/N
\\	c & -b
\end{pmatrix}
\bigg| \ a,b,c \in \mathbb{Z}\right\}.
\] This is an even lattice of level $4N$ and discriminant $2N$. $L'/L$ can be identified with $\mathbb{Z}/N\mathbb{Z}$ with discriminant form $x \mapsto x^{2}/4N$. The dual lattice is given by
 \[
  L'= \left\{\begin{pmatrix}
	b/2N & -a/N
\\	c & -b/2N
\end{pmatrix}
\Bigg| \ a,b,c \in \mathbb{Z}\right\}.
\]
We let $\mathrm{GL}_{2}(\mathbb{Q})$ act isometrically on $V$ via conjugation. This gives rise to the isomorphism $\mathrm{SL}_{2}(\mathbb{Q}) \cong \mathrm{Spin}(V)$.  For $m$ be an exact divisor of $N$ we denote the corresponding Atkin-Lehner involution on $\Gamma_{0}(N)$ by $ W_{m}^{N}$. 

In signature $(2,1)$, $\mathrm{Gr}(V(\mathbb{R}))$ is the set of negative lines which we identify with one component of a two-sheet hyperboloid. We fix an isotropic vector $l \in V$ and then call
\[ V_{-1} \coloneqq \left\{ v_{-1} \in V(\mathbb{R}) \mid (v_{-1},v_{-1}) = -1, \ (v_{-1},l)<0 \right\} \]
the hyperboloid model, where we form a bijection from $V_{-1}$ to $\mathrm{Gr}(V(\mathbb{R}))$ via the map $v_{-1} \mapsto \mathbb{R}v_{-1}.$ We also form a bijection from  $\mathbb{H}$ to $V_{-1}$ via the map  \[ \lambda(x+iy) \coloneqq \frac{1}{\sqrt{2N}y}\begin{pmatrix} -x &  x^{2}+y^{2} \\ -1 & x \end{pmatrix}. \] As is standard, we will often abuse notation, and set $z=x+iy \in \mathbb{H}$ but also denote by $z$ its identifications in $V_{-1}$ and $\mathrm{Gr}(V(\mathbb{R}))$. We have an orthonormal basis of $V(\mathbb{R})$ consisting of $e_{1} \coloneqq \frac{1}{\sqrt{2N}} \begin{psmallmatrix} 0 & 1 \\ 1 & 0 \end{psmallmatrix}, e_{2} \coloneqq \frac{1}{\sqrt{2N}} \begin{psmallmatrix} 1 & 0 \\ 0 & -1 \end{psmallmatrix}, e_{3} \coloneqq \frac{1}{\sqrt{2N}} \begin{psmallmatrix} 0 & 1 \\ -1 & 0 \end{psmallmatrix}$ and we fix an oriented basis $b_{1}(z) \coloneqq g_{z}.e_{1}, \ b_{2}(z) \coloneqq g_{z}.e_{2}$ and $b_{3}(z) \coloneqq g_{z}.e_{3}$ where $g_{z} \coloneqq \begin{psmallmatrix} \sqrt{x} & x/\sqrt{y} \\ 0 & 1/\sqrt{y} \end{psmallmatrix}$. If $\lambda \in V(\mathbb{R}))$ then $\lambda = \sum \lambda_{i}(z)b_{i}(z)$ where $\lambda_{i}(z) \coloneqq \frac{(\lambda, b_{i}(z))}{(b_{i}(z),b_{i}(z))}$.\\

We have the modular curve $Y_{0}(N) := \Gamma_{0}(N)\backslash \mathbb{H}$. 
When $N$ is square-free there are $\sigma_0(N)$ cusps, and can be represented by $W_{d}^{N} \infty$, with $d$ running over the divisors of $N$. The cusps can be identified with the set of isotropic lines $\mathrm{Iso}(V)$ in $V$ by the map $\sigma_{\mathbb{P}^{1}(\mathbb{Q})}^{\textrm{Iso}(V)}:\mathbb{P}^{1}(\mathbb{Q}) \rightarrow \textrm{Iso}(V)$ given by $(m/n) \mapsto \textrm{span}\left(
\begin{smallmatrix}
	-m n & m^2
	\\ -n^2 & m n
\end{smallmatrix}\right)$. Each of these lines $l'$ can be uniquely represented by a primitive isotropic vector in $L$ up to sign. We choose our primitive vectors $l'$ so that $\textrm{sgn}((-l', g_{z}.e_{3})=1$. Then the cusps $\infty$ and $0$ correspond to $\Gamma_{0}(N)$-classes of $l_{\infty} :=  \left(
\begin{smallmatrix}
	0 & 1/N
	\\ 0 & 0
\end{smallmatrix}\right)$ and $l_{0} :=  \left(
\begin{smallmatrix}
	0 & 0
	\\ -1 & 0
\end{smallmatrix}\right)$.\\

Let $\lambda \in V, Q(\lambda) >0$. Then we have the associated geodesic
\[
 D_{\lambda} := \left\{ z \in \mathrm{Gr}(V(\mathbb{R})) \mid z \perp \lambda \right\} 
 \]
or in the upper half plane model
\[
D_{\lambda} = \left\{z \in \mathbb{H} \ | \ cN|z|^{2} - b x + a = 0\right\}.
\] 
These are either vertical half lines (when $\lambda \perp l_{\infty}$) or semi-circles (when $\lambda \not\perp l_{\infty}$). Let $\Gamma_{\lambda} \coloneqq \left\{ \gamma \in \Gamma_{0}(N) \mid \gamma.\lambda = \lambda \right\}$ be the stabiliser of $\lambda$ in $\Gamma_{0}(N)$ and define the cycle $Z(\lambda)$ be the image of the quotient $\Gamma_{\lambda} \backslash D_{\lambda}$ in the modular curve $Y_{0}(N)$. We orient the cycles counterclockwise if $a>0$ and clockwise if $a<0$, in particular $D_{-\lambda}=-D_{\lambda}$. For $\mathcal{D} \in \mathbb{Z}, h \in L'/L$, we write
\begin{equation}\label{finite orbits}
 L_{\mathcal{D},h} \coloneqq \left\{\lambda \in L' \mid Q(\lambda) = \mathcal{D}/4N,  \ \lambda \equiv h \pmod {L} \right\}. 
\end{equation}
It is well known that if $\mathcal{D} \neq 0$ then there are only finitely many $\Gamma_0(N)$-orbits of $L_{\mathcal{D},h}$.

Let $\Delta \in \mathbb{Z}$ be a fundamental discriminant and $r \in \mathbb{Z}$ such that $\Delta \equiv r^{2} \pmod{4N}$ and set $\lambda = \left(
\begin{smallmatrix}
	b/2N & -a/N
\\	c & -b/2N
\end{smallmatrix}\right) \in L'.$ Let $n$ be any integer that is coprime to $\Delta$ and representable by a binary quadratic form $[N_{1}a,b,N_{2}c]$ with $N_{1}N_{2}=N$ and $N_{1}, N_{2} >0$ i.e. $n=[N_{1}a,b,N_{2}c](x,y)$ for some $x,y \in \mathbb{Z}$. If $\Delta$ is such that \begin{enumerate}
\item $4NQ(\lambda)/\Delta \equiv s^{2}\pmod {4N}$ for some $s \in \mathbb{Z}$,
\item $\gcd(a,b,c,\Delta) = 1$,
\end{enumerate} then we define the {\it generalised genus character} as $\chi_{\Delta}(\lambda) : = \left(\frac{\Delta}{n} \right)$ otherwise $\chi_{\Delta}(\lambda) \coloneqq 0$. $\chi_{\Delta}$ is invariant under the action of $\Gamma_{0}(N)$ and the Atkin-Lehner involutions. Furthermore, $\chi_{\Delta}(\lambda)$ only depends on $\lambda \in L'$ modulo $\Delta L$. We define the twisted Weil representation $\tilde{\rho}_{L}$ to be equal to $\rho_{L}$ if $\Delta >0$ and equal to $\overline{\rho}_{L}$ if $\Delta <0$. We will often use the notation $\rho\coloneqq \tilde{\rho}_{L}$.  \\

\begin{definition}\label{twisted cycle}
Let $\Delta$ be a fundamental discriminant. Let $m \in \mathbb{Z} - \mathrm{sgn}(\Delta)Q(h)$ with $m<0$. We set $d \coloneqq 4N\mathrm{sgn}(\Delta)m$ and 
\[
Z_{\Delta,r}(m,h) \coloneqq \sum_{\lambda \in L_{-d\Delta,rh} / \Gamma_{0}(N)} \chi_{\Delta}(\lambda)Z(\lambda). 
\]
We call $Z_{\Delta,r}(m,h)$ a twisted special cycle. Let $f \in H_{k, \tilde{\rho}_{L^{-}}}$. Then we set
\[ 
Z_{\Delta,r}(f) \coloneqq \sum_{h \in L'/L} \sum_{\substack{m \in \mathbb{Z} - \mathrm{sgn}(\Delta)Q(h) \\ m <0 }}
c^{+}(m,h)Z_{\Delta, r}(m,h),
\]
the associated twisted cycle. We write $Z'_{\Delta,r}(f)$ for its pre-image in $\mathbb{H}$.
\end{definition}

\subsection{Kernel Functions}

We set
 \begin{align*}
p_{z}(\lambda) & {} \coloneqq {} -(\lambda,\lambda(z))= -(\lambda,g_{z}.e_{3}) =\lambda_{3}(z) = \frac{-1}{\sqrt{2N}y}(cN|z|^{2}-bx+a), \\
q_{z}(\lambda) & {} \coloneqq {} y(\lambda, g_{z}.(e_{1}+ie_{2})) = y(\lambda_{1}(z)+i\lambda_{2}(z)) = \frac{-1}{\sqrt{2N}}(cNz^{2} -bz+a).
\end{align*}
Using the explicit isometry $\sigma_{z}: V(\mathbb{R}) \rightarrow \mathbb{R}^{2,1}$ given by 
\[
\sigma_{z}(\lambda) \coloneqq  \frac{1}{\sqrt{2}} \left(\lambda_{1}(z), \lambda_{2}(z), \lambda_{3}(z) \right)
\] we can check that $q_{z}(\lambda)^{k-1}p_{z}(\lambda)$ and $\left(  q_{z}(\lambda)/y^{2} \right)^{k}$ are harmonic polynomials of degree $(k-1,1)$ and $(k,0)$ respectively.
\begin{definition}\label{Kernel Function}
Let $h \in L'/L$. For $k \geq 1$ we define the kernel functions as follows:
 \begin{align*}
\theta_{\Delta,r,h,k}(\tau,z){} \coloneqq {} & v^{3/2}  \sum_{\substack{\lambda \in L +rh \\ Q(\lambda) \equiv \Delta Q(h)(\Delta) }} \chi_{\Delta}(\lambda)  p_{z}(\lambda)  q_{z}(\lambda)^{k-1} e\left(\frac{Q(\lambda)}{|\Delta|}u + \frac{Q_{z}(\lambda)}{|\Delta|} iv \right),
\\ \theta^{*}_{\Delta,r,h,k}(\tau,z) {} \coloneqq {} & v^{1/2}  \sum_{\substack{\lambda \in L +rh \\ Q(\lambda) \equiv \Delta Q(h)(\Delta) }} \chi_{\Delta}(\lambda) \left(\frac{q_{z}(\lambda)}{y^{2}}\right)^{k} e\left(\frac{Q(\lambda)}{|\Delta|}u + \frac{Q_{z}(\lambda)}{|\Delta|} iv \right).
\end{align*}
The $\mathbb{C}[L'/L]$-valued versions are
$\Theta_{\Delta,r,k}(\tau,z) \coloneqq \sum_{h \in L'/L}\theta_{D,r,h,k}(\tau,z)\mathfrak{e}_{h}$ and $\Theta^{*}_{\Delta,r,k}(\tau,z) \coloneqq \sum_{h \in L'/L}\theta^{*}_{\Delta,r,h,k}(\tau,z)\mathfrak{e}_{h}$.
\end{definition}
Both these functions have exponential decay as $v \rightarrow \infty$, uniformly in $u$. 
We have $\overline{\Theta_{\Delta,r,k}(\tau,z)} = \Theta_{\Delta,r,k}(-\overline{\tau},-\overline{z})$ and $\overline{\Theta^{*}_{\Delta,r,k}(\tau,z)} = \Theta^{*}_{\Delta,r,k}(-\overline{\tau},-\overline{z})$. 

\begin{remark}\label{how kernel compares} For $k$ even and with $N=1$ and $\Delta = 1$ these kernel functions were considered in \cite[Section~1]{BKZ} and \cite[(1.6)]{BKV}. In the case $k=1$ they were used in \cite{Hoevel} and \cite{AGOR}. 
\end{remark}

\begin{proposition}\label{kernel tau transformation}
As a function of $\tau$ we have $\Theta_{k}(\tau,z) \in A_{k-3/2, \tilde{\rho}_{L}}$ and  $\Theta^{*}_{k}(\tau,z) \in A_{k+1/2, \tilde{\rho}_{L}}$.
\end{proposition}
\begin{proof} The polynomials are harmonic of degree $(k-1,1)$ and $(k,0)$ respectively. The result follows from \eqref{theta transform general} and the results of \cite[Section~3]{AE}. See also \cite[Theorem~3.1]{AGOR} and \cite[Section~3]{AS}.
 \end{proof}

As $\Gamma_{0}(N)$ acts trivially on $L'/L$ and using $\chi_{\Delta}(\gamma.\lambda) = \chi_{\Delta}(\lambda)$ an easy calculation gives the modularity in $z$: 

\begin{proposition}\label{z transformation}
For all $\gamma \in \Gamma_{0}(N)$ we have 
\begin{align*}
\Theta_{k}(\tau,\gamma.z)= j(\gamma,z)^{2-2k}\Theta_{k}(\tau,z), \qquad  \Theta^{*}_{k}(\tau,-\overline{\gamma.z})=j(\gamma,z)^{2k}\Theta^{*}_{k}(\tau,-\overline{z}).
\end{align*}
\end{proposition} 

As the genus character is invariant under Atkin-Lehner involutions we also see

\begin{proposition}\label{atkin lehner transformation}
We have  
\[
  \Theta_{k}(\tau,W_{m}^{N}.z)= j(W_{m}^{N},z)^{2-2k}\sum_{h \in L'/L}\theta_{W_{m}^{N}.h,k}(\tau,z)\mathfrak{e}_{h}.
  \] 
\end{proposition}

\section{The Singular Theta Lift}

For $f \in H_{3/2-k, \overline{\rho}}$ the decomposition in \eqref{Decomposition} becomes 
\begin{equation}\label{Simple Expansion}
 f(\tau) = \sum_{h \in L'/L}\sum_{m \in \mathbb{Z} - \mathrm{sgn}(\Delta) Q(h)}c(m,h,v)e(m\tau)\mathfrak{e}_{h} ,
\end{equation}
with $c(m,h,v) = c^{+}(m,h) + c^{-}(m,h)\Gamma(k-1/2,-4 \pi mv)$ and $c^{+}(m,h)=0$ for $m < -n_0$. 

\begin{definition}\label{The Lift}
Let $f \in H_{3/2-k, \overline{\rho}}$. Then we define 
\begin{align*} \Phi_{\Delta,r,k}(z,f) {} \coloneqq {} & \left( f(\tau),v^{k-3/2}\overline{\Theta_{\Delta,r,k}(\tau,z)} \right)_{3/2-k,\overline{\rho}}^{\mathrm{reg}}  =  \int_{\tau \in \mathcal{F}}^{\mathrm{reg}} \left\langle f(\tau), \overline{\Theta_{\Delta,r,k}(\tau,z)}\right\rangle \frac{dudv}{v^2}.
\end{align*}
\end{definition} 

We first check the regularised integral converges on all of $\H$, including the singularities. We follow the ideas in \cite[Section~6]{Borcherds}, \cite[Proposition~2.8]{B-Habil}, and \cite[Proposition~5.6]{BF-Duke}.

\begin{theorem}\label{Pointwise} The regularised Petersson scalar product $\Phi_{D,r,k}(z,f)$ converges pointwise for any $z=x+iy \in \mathbb{H} \cong \mathrm{Gr}(V(\mathbb{R}))$.
\end{theorem}
\begin{proof} It suffices to consider the convergence of 
\begin{equation}\label{Rectangle} \int_{v=1}^{\infty} \int_{u= -1/2}^{1/2} \left\langle f^{+}(\tau),\overline{\Theta_{\Delta,r,k}(\tau,z)}\right\rangle \frac{du dv}{v^{2}}. 
\end{equation} 
We plug in the expansions given in \eqref{Simple Expansion}  and Definition \ref{Kernel Function} and carry out the integration over $u$ (noting $p_{z}(0) = 0$) to obtain
\begin{equation} \label{Simplified Integral}
\int_{v=1}^{\infty}  \sum_{h \in L'/L}  \sum_{\substack{\lambda \in L +rh \\ Q(\lambda) \equiv \Delta Q(h)(\Delta) \\ \lambda \neq 0 }} c^{+}\left(\tfrac{-Q(\lambda)}{|\Delta|},h \right) \chi_{\Delta}(\lambda) p_{z}(\lambda) q_{z}(\lambda)^{k-1} e\left( \tfrac{-2 Q(\lambda_{z})}{|\Delta|}iv\right)v^{-1/2} dv.
\end{equation}
Using the identities $|p_{z}(\lambda)| = \sqrt{-2Q(\lambda_{z})}$ and $|q_{z}(\lambda)| = y\sqrt{2Q(\lambda_{z^{\perp}})}$ and estimating $v^{-1/2} \leq 1$ for $1 \leq v \leq \infty$, it remains to check that the following converges:
\begin{equation}\label{Pre case}  \frac{|\Delta|}{2\sqrt{2} \pi}\sum_{h \in L'/L } \sum_{\substack{\lambda \in L +rh \\ Q(\lambda) \equiv \Delta Q(h)(\Delta) \\ \lambda \neq 0 }} \left|c^{+}\left(\frac{-Q(\lambda)}{|\Delta|},h \right)\right| \frac{(y\sqrt{2Q(\lambda_{z^{\perp}})})^{k-1}}{\sqrt{- Q(\lambda_{z})}}e\left( \frac{-2 Q(\lambda_{z})i}{|\Delta|}\right) 
\end{equation}
We now split the sum into three parts and check each converges. 

\textbf{Case $Q(\lambda) = 0$:} We know that $Q(\lambda_{z^{\perp}}) = -Q(\lambda_{z}) = Q_{z}(\lambda)/2$ so we are left with a subseries of a convergent theta series (for the positive definite quadratic form $Q_{z}(\lambda)$).
 
\textbf{Case $Q(\lambda) < 0$:} We have from \cite[Lemma~3.4]{BF-Duke}, \cite[Lemma~1.49]{Hoevel} that there exists a constant $C>0$ such that $\left|c^{+}\left(\tfrac{-Q(\lambda)}{|\Delta|},h \right)\right| \leq C e^{C\sqrt{-Q(\lambda)}}$. We also observe in this case that $Q_{z}(\lambda) \geq -Q(\lambda)$, $Q_{z}(\lambda) > Q(\lambda_{z^{\perp}})$ and $-2Q(\lambda_{z}) = -Q(\lambda_{z}) - Q(\lambda) + Q(\lambda_{z^{\perp}}) > Q_{z}(\lambda)$. We use these inequalities to again bound \eqref{Pre case} with a theta series in terms of $Q_{z}(\lambda)$. 

\textbf{Case $Q(\lambda) > 0$:} We remember there are only finitely many $c^{+}(m,h) \neq 0$ with $m < 0$. For each $m= -Q(\lambda)/|\Delta|$ it then suffices to check that
\begin{equation}\label{tricky case}
\frac{|\Delta|}{2\sqrt{2} \pi} \sum_{\substack{\lambda \in L +rh \\ -Q(\lambda) = |\Delta|m \\ Q(\lambda_{z} ) \neq  0}} \frac{(y\sqrt{2Q(\lambda_{z^{\perp}})})^{k-1}}{\sqrt{- Q(\lambda_{z})}}e\left( \frac{-2 Q(\lambda_{z})i}{|\Delta|}\right)
\end{equation} converges. We know from \cite[p.50]{B-Habil} that for any $C \geq 0$ and any compact $U \subset Gr(L)$ the set
\begin{equation}\label{finitely small}
 \left\{\lambda \in L' \ |  \ -Q(\lambda) = |\Delta|m, \ \exists z' \in U \ \textrm{with} \ -Q(\lambda_{z'}) \leq C \right\}
\end{equation}
is finite. Let $\lambda \in L'$ be such that $-Q(\lambda) = |\Delta|m$ and $Q(\lambda_{z}) \neq 0$. Then there exists an $\epsilon > 0$ such that $-Q(\lambda_{z}) > \epsilon$ for all $\lambda$. This means we have $Q_{z}(\lambda) \geq Q(\lambda_{z^{\perp}})$, $-Q(\lambda_{z}) > \epsilon$ and $-2Q(\lambda_{z}) = -Q(\lambda) + Q_{z}(\lambda) = |\Delta|m + Q_{z}(\lambda)$. So once again we can use these to bound \eqref{Pre case} with a theta series in terms of $Q_{z}(\lambda)$. \end{proof}

\begin{theorem}\label{Singularities}
\begin{enumerate}
\item $\Phi_{\Delta,r,k}(z,f)$ has weight $2-2k$ for $\Gamma_{0}(N)$.
\item $\Phi_{\Delta,r,k}(z,f)$ is a smooth function on $\mathbb{H}\backslash Z'_{\Delta,r}(f)$.
\item  $\Phi_{\Delta,r,k}(z,f)$ has singularities along  $Z'_{\Delta,r}(f)$. More precisely, for a point $z_{0} \in \mathbb{H}$ exists an open neighbourhood $U \subset \mathbb{H}$ so that the function
\[  \Phi_{\Delta,r,k}(z,f) - \sqrt{\tfrac{|\Delta|}{2}}   \sum_{\substack{h \in L'/L \\ m \in \mathbb{Z} - \mathrm{sgn}(\Delta)Q(h) \\ m < 0 }}c^{+}\left(m,h \right) \sum_{\substack{\lambda \in L_{-d\Delta,rh} \\ \lambda \perp z_{0}}} \tfrac{\chi_{\Delta}(\lambda) q_{z}(\lambda)^{k-1} (\lambda,v(z))}{|(\lambda,v(z))|} 
\]
can be continued to a smooth function on $U$. (Here we let the term on the right hand side vanish if $z \in U, (\lambda,v(z))=0$). 
\end{enumerate}
\end{theorem}

\begin{proof}
Using Theorem \ref{z transformation} the first statement is clear. As in Theorem \ref{Pointwise} we see that the integrals of the $f^{-}$ part and the integral over the compact region $\mathcal{F}_{1}$ converge absolutely and therefore define smooth functions on $\mathbb{H}$ and do not contribute to the singularities. It then remains to look at  \eqref{Simplified Integral}. 

We first consider the case where $Q(\lambda) \leq 0$. The arguments in Theorem \ref{Pointwise} can be adapted to show local uniform convergence if for any point $z_{0} \in \mathbb{H}$ there exists an open subset $U \subset \mathbb{H}$ and a constant $\epsilon >0$ such that $Q(\lambda_{z}) < \epsilon$ for all $\lambda \in L', \lambda \neq 0, Q(\lambda) \leq 0$ and $z \in U$. We know this to be true using \cite[Equation~3.17]{Hoevel}. 

So only the terms where $Q(\lambda)>0$ contribute to the singularities. Fix $z_{0} \in \mathbb{H}$. Using \eqref{Simplified Integral} and \eqref{finite orbits} we are then left to consider\footnote{Using the $\approx$ notation from \cite[Theorem~2.12]{B-Habil}}:
\begin{multline*}  \Phi_{\Delta,r,k}(z,f) \approx  \sum_{h \in L'/L}  \sum_{\substack{m \in \mathbb{Z} - \mathrm{sgn}(\Delta)Q(h) \\ m < 0 }}c^{+}\left(m,h \right) \\   \times \int_{v=1}^{\infty} \sum_{\lambda \in L_{-d \Delta,rh}}\chi_{\Delta}(\lambda) p_{z}(\lambda) q_{z}(\lambda)^{k-1} e\left( \frac{-2 Q(\lambda_{z})}{|\Delta|}iv\right)  v^{-1/2} dv.
\end{multline*}
We now split the sum over $\lambda \in L_{-d \Delta,rh}$ into two sums, one over $\lambda \perp z_{0}$ and one over $\lambda \not\perp z_{0}$.

For $\lambda \not\perp z_{0}$ we need to check that for $z_{0} \in \mathbb{H}$ there exists an open subset $U \subset \mathbb{H}$ (with compact closure $\overline{U} \subset \mathbb{H}$) and a constant $\epsilon >0$ such that $Q(\lambda_{z}) < \epsilon$ for all $\lambda \in L_{-d \Delta, rh}, \lambda \not\perp z_{0}$ and $z \in U$. This is true using \eqref{finitely small} and noting that $\lambda \not\perp z_{0}$ means we can choose a neighbourhood $U$ of $z_{0}$ small enough such that $Q(\lambda_{z}) \neq 0$.
 
Finally, we look at the sum over $\lambda \in L_{-d\Delta,rh}$ where $\lambda \perp z_{0}$. We first notice that $\lambda \perp z_{0}$ means that $Q(\lambda_{z_{0}})=0$. We can then use \eqref{finitely small} to see we actually have a finite sum over $\lambda \in L_{-d\Delta,rh},\lambda \perp z_{0}$. We now look at the remaining integral. We have
\[ \int_{v=1}^{\infty} p_{z}(\lambda) q_{z}(\lambda)^{k-1} e\left( \frac{-2 Q(\lambda_{z})}{|\Delta|}iv\right)  v^{-1/2} = \sqrt{\frac{|\Delta|}{2 \pi}}\frac{p_{z}(\lambda)q_{z}(\lambda)^{k-1}}{|2Q(\lambda_{z})|}\Gamma\left(\frac{1}{2}, \frac{-4 \pi Q(\lambda_{z})}{|\Delta|} \right).
\]
When $\lambda_{z} =0$ this has a singularity of type
\[ \sqrt{\frac{|\Delta|}{2}}\frac{(\lambda,v(z))}{|(\lambda,v(z))|}f_{z}(\lambda)^{k-1}
\]
as we know that $\Gamma(1/2, -4 \pi Q(\lambda_{z})/|\Delta|) = \Gamma(1/2) + \mathcal{O}(|Q(\lambda_{z})|)$ as $\lambda_{z} \rightarrow 0$. The integral vanishes when $-2Q(\lambda_{z}) = p_{z}(\lambda)=(\lambda,v(z))=0$ so this is the zero contribution to the singularities when $(\lambda,v(z))=0$. So finally we have the required result
\begin{align*}  \Phi_{\Delta,r,k}(z,f) \approx_{U} & \sqrt{\frac{|\Delta|}{2}} \sum_{h \in L'/L}  \sum_{\substack{m \in \mathbb{Z} - \mathrm{sgn}(\Delta)Q(h) \\ m < 0 }}c^{+}\left(m,h \right) \sum_{\substack{\lambda \in L_{-d\Delta,rh} \\ \lambda \perp z_{0}}}\chi_{\Delta}(\lambda) \frac{(\lambda,v(z))}{|(\lambda,v(z))|} f_{z}(\lambda)^{k-1}. \qedhere
\end{align*} 
\end{proof}

We find the wall crossing formula as we move between Weyl chambers. We follow \cite[Section~6]{Borcherds} and \cite[Section~3.1]{B-Habil}. Let $W \subset \mathbb{H}$ be a Weyl chamber and let $\lambda \in L'$. Then we say $(\lambda,W)<0$ if $(\lambda, w) < 0$ for all $w \in W \subset \mathbb{H}$. We will denote $\Phi_{W_{1}}(z)$ and $\Phi_{W_{2}}(z)$ for the restrictions of $\Phi_{\Delta,r,k}(z,f)$ to two adjacent Weyl chambers $W_{1}$ and $W_{2}$. The restrictions $\Phi_{W_{1}}(z)$ and $\Phi_{W_{2}}(z)$ can both be extended to real analytic functions on $\overline{W_{1}} \cup \overline{W_{2}}$ and we write $W_{12} \coloneqq \overline{W_{1}} \cap \overline{W_{2}}$ for the ``wall" dividing $W_{1}, W_{2}$. 

\begin{theorem}[The wall crossing formula]\label{wall cross theorem}
The difference $\Phi_{W_{1}}(z)-\Phi_{W_{2}}(z)$ is given by
\[   2\sqrt{2|\Delta|} \sum_{h \in L'/L}  \sum_{\substack{m \in \mathbb{Z} - \mathrm{sgn}(\Delta)Q(h) \\ m < 0 }}c^{+}\left(m,h \right) \sum_{\substack{\lambda \in L_{-d\Delta,rh} \\ \lambda \perp W_{12} \\ (\lambda,W_{1})<0}}\chi_{\Delta}(\lambda) q_{z}(\lambda)^{k-1}.
\]
\end{theorem} 
\begin{proof}
Using Theorem \ref{Singularities} we know that $\Phi_{\Delta,r,k}(z,f)$ has a singularity of type 
\begin{equation}\label{wall cross part}  \sqrt{\frac{|\Delta|}{2}} \sum_{h \in L'/L}  \sum_{\substack{m \in \mathbb{Z} - \mathrm{sgn}(\Delta)Q(h) \\ m < 0 }}c^{+}\left(m,h \right) \sum_{\substack{\lambda \in L_{-d\Delta,rh} \\ \lambda \perp W_{12}}}\chi_{\Delta}(\lambda) \frac{(\lambda,v(z))}{|(\lambda,v(z))|} q_{z}(\lambda)^{k-1}
\end{equation}
along $W_{12}$. We observe that the sums over $\lambda$ and $-\lambda$ are the same because $\chi_{\Delta}(-\lambda) = \\ (-1)^{(1-\mathrm{sgn}(\Delta))/2}\chi_{\Delta}(\lambda)$, $c^{+}(m,h) = (-1)^{3/2-k+(\mathrm{sgn}(\Delta))/2}c^{+}(m,-h)$ and $p_{z}(-\lambda)q_{z}(-\lambda)^{k-1} = \\ (-1)^{k}p_{z}(\lambda)q_{z}(\lambda)^{k-1} $. We can then rewrite \eqref{wall cross part} as a sum over elements with $(\lambda,W_{1}) < 0$. This means we pick up a factor of $2$ and also another factor of $2$ from the jump of size $2$ arising from $(\lambda,v(z))/|(\lambda,v(z))|$.
\end{proof}

\subsection{Locally Harmonic}

We will show that the singular theta lift is harmonic away from the singularities $Z'_{\Delta,r}(f)$. Using \cite{Shintani}, Lemma~1.5 we see 

\begin{proposition}\label{Laplacians} We have 
\[ 4\Delta_{k-3/2, \tau}\Theta_{\Delta,r,k}(\tau,z) = \Delta_{2-2k,z}\Theta_{\Delta,r,k}(\tau,z) + (6-4k)\Theta_{\Delta,r,k}(\tau,z).
\]
\end{proposition}

\begin{theorem}\label{Locally Harmonic}
For $f \in H_{3/2-k,\overline{\rho}}$ and $z \in \mathbb{H}\backslash Z_{\Delta,r}(f)$, then $\Delta_{2-2k}\Phi_{\Delta,r,k}(z,f) = 0$ and $\Phi_{\Delta,r,k}(z,f)$ is also real analytic on $\mathbb{H}\backslash Z'_{\Delta,r}(f)$.
\end{theorem}
\begin{proof}
We apply Proposition \ref{Laplacians} and obtain
\begin{align*}
 \Delta_{2-2k}\Phi_{\Delta,r,k}(z,f) =& 4\lim_{t \to \infty} \int_{\tau \in \mathcal{F}_{t}}\left\langle   \Delta_{k-3/2,\tau}\Theta_{\Delta,r,k}(\tau,z), v^{3/2-k} \overline{f(\tau)} \right\rangle v^{k-3/2}\frac{du dv}{v^{2}} \\ &  + 4(k-3/2) \lim_{t \to \infty} \int_{\tau \in \mathcal{F}_{t}}\left\langle  \Theta_{\Delta,r,k}(\tau,z), \overline{f(\tau)} \right\rangle \frac{du dv}{v^{2}}.
\end{align*} 
Arguing as in \cite[Section~4]{B-Habil}, using adjointness of the Laplace operator we obtain
\begin{align*}
 \Delta_{2-2k}\Phi_{\Delta,r,k}(z,f) =& 4\lim_{t \to \infty} \int_{\tau \in \mathcal{F}_{t}}\left\langle   \Theta_{\Delta,r,k}(\tau,z),   \Delta_{k-3/2,\tau}\left(v^{3/2-k} \overline{f(\tau)}\right) \right\rangle v^{k-3/2}\frac{du dv}{v^{2}} \\ &  + 4(k-3/2) \lim_{t \to \infty} \int_{\tau \in \mathcal{F}_{t}}\left\langle  \Theta_{\Delta,r,k}(\tau,z), \overline{f(\tau)} \right\rangle \frac{du dv}{v^{2}}.
\end{align*} 
Here we used that in the limit the boundary terms vanish. But $\Delta_{k-3/2,\tau}\left(v^{3/2-k} \overline{f(\tau)}\right)= (3/2-k) v^{3/2-k} \overline{f(\tau)}$, since $f$ is harmonic, and the claim follows. 
\end{proof}

\section{Partial Poisson Summation}
In Section \ref{The Fourier Expanison Section} we will compute the Fourier expansion of our lift using a Rankin-Selberg style unfolding trick. For this we will need to rewrite the kernel function as a Poincar\'e series. Useful references are \cite[Section~5]{Borcherds},  \cite[Section~2]{B-Habil}, and \cite[Section~4]{BO-Annals}. \\

Let $l \in L$ be a primitive isotropic vector. We define a $1$-dimensional positive definite space $W_{l} \coloneqq l^{\perp}/l$, (equipped with the same quadratic form) and a sublattice, 
\[ 
K=K_{l} \coloneqq (L \cap l^{\perp})/(L \cap l) 
\]
The dual lattice is $K'=K'_{l} = (L' \cap l^{\perp})/(L' \cap l)$. We know there exists a vector $l' \in L'$ such that $(l,l')=1$ and then $K_{l} = L \cap l'^{\perp} \cap l^{\perp}.$
For $\lambda \in V(\mathbb{R})$, we write $\lambda_{K}$ for the orthogonal projection onto $K \otimes \mathbb{R}$. If $\lambda \in L'$, then $\lambda_{K} \in K'$. We will now assume that $(l,L) = \mathbb{Z}$, in which case we can choose $l'$ to be isotropic. This holds for $l_{\infty}$, but also for any $l$ when $N$ is square-free. We then have \begin{subequations}\label{K decompostion}
\begin{align} L = & K_{l} \oplus \mathbb{Z}l' \oplus \mathbb{Z}l,
\\ V(\mathbb{R}) = & (K_{l} \otimes_{\mathbb{Z}} \mathbb{R}) \oplus \mathbb{R}l' \oplus \mathbb{R}l,
\end{align}
\end{subequations} and $ K_{l}'/K_{l} \cong L'/L$. We write $w^{\perp}$ for the orthogonal complement of $l_{z^{\perp}}$ in $z^{\perp}$ and denote the component of any $\lambda \in V(\mathbb{R})$ in $w^{\perp}$ as $\lambda_{w^{\perp}}$. We have
\[ V(\mathbb{R}) = z \oplus z^{\perp} = \mathbb{R}l_{z} \oplus \mathbb{R}l_{z^{\perp}} \oplus w^{\perp}.
\] We will also use the vectors
\begin{equation}\label{Mu}  \mu =\mu(z) := -l' + \frac{l_{z}}{2 (l_{z},l_{z})} + \frac{l_{z^{\perp}}}{2 (l_{z^{\perp}}, l_{z^{\perp}})}, \quad \mathfrak{w}^{\perp} := (l,b_{2}(z))b_{1}(z) - (l,b_{1}(z))b_{2}(z).
\end{equation}
We observe that $(l_{z^{\perp}},\mathfrak{w}^{\perp})= 0$, thus $\mathfrak{w}^{\perp}$ spans the one dimensional space $w^{\perp}$. In general $\lambda, \lambda_{K}, \lambda_{w^{\perp}}$ are not the same vector. However if $\lambda \in V(\mathbb{R}) \cap l^{\perp}$, then $(\lambda, \lambda) = (\lambda_{K}, \lambda_{K}) = (\lambda_{w^{\perp}}, \lambda_{w^{\perp}})$.  Finally, if $\lambda \in V(\mathbb{R}) \cap l^{\perp}$, then $(\lambda, \lambda)/2 = (\lambda_{w^{\perp}}, \lambda_{w^{\perp}})/2 = (\lambda,\mathfrak{w}^{\perp})^{2}/(2Q_{z}(l)).$ We also have the folowing identities.  \begin{lemma}\label{mu facts}
~
\begin{enumerate}
\item $\mu \in V(\mathbb{R} \cap l^{\perp}) = (K \otimes_{\mathbb{Z}} \mathbb{R}) \oplus \mathbb{R}l'$,
\item $\mu = \mu_{K} +(\mu,l')l$,
\item $(\mu,l)=(\mu_{K},l) = 0$,
\item $(\mu, \mu) = (\mu_{K}, \mu_{K}) = (\mu_{w^{\perp}}, \mu_{w^{\perp}})$,
\item $\mu_{w^{\perp}} = (\mu_{K})_{w^{\perp}} = -l^{'}_{w^{\perp}}$,
\item $(\mu, \mathfrak{w}^{\perp}) = (\mu_{K},\mathfrak{w}^{\perp}) = (-l',\mathfrak{w}^{\perp})$,
\item $(\mu,\mu)/2 = (\mu_{K}, \mu_{K})/2=  - (l',l_{z^{\perp}}-l_{z})/(2Q_{z}(l))$,
\item If $\lambda \in K \otimes \mathbb{R}$, then $(\lambda,\mu)=(\lambda,\mu_{K}) = (\lambda,l_{z^{\perp}}-l_{z})/(2Q_{z}(l))$.
\end{enumerate}
\end{lemma}
\subsection{Fourier Transforms}
We will need several Fourier transforms. We will make use of the Hermite polynomails $H_{n}(x) \coloneqq (-1)^{n}e^{x^{2}}\frac{d^{n}}{d x^{n}}\left( e^{-x^{2}} \right)$, see \cite[Section~10.13]{EMOT-2}. We let the Fourier transform over $\mathbb{R}$ be defined as 
$ \hat{f}(\xi) \coloneqq \int_{\mathbb{R}} f(x) e^{2 \pi i \xi x} dx$.

\begin{lemma}\label{Easy Fouriers}The Fourier transform of
\begin{enumerate}
	\item $f(x-a)$ is $e^{2 \pi i a \xi}\hat{f}(\xi)$,
	\item $x f(x)$ is $\frac{d}{d\xi}\hat{f}(\xi)/2 \pi i$,
	\item $f(a x)$ is $|a|^{-1}\hat{f}(\xi/a) $,
	\item  $x^n e^{- \pi x^2}$ is
 $\left(\frac{i}{2 \sqrt{\pi}}\right)^{n}H_{n}(\sqrt{\pi} \xi) e^{- \pi \xi^2}$.
\end{enumerate}
\end{lemma}
\begin{proof}
(1)-(3) are of course standard. (4) can be derived from \cite[Lemma~4.5]{FM-res}. 
\end{proof}

\begin{lemma}\label{Final Fourier}
Let $A,B,C,D,E,F,G \in \mathbb{C}, \mathrm{Im}(A) >0$. Then the Fourier transform of $(G+F x)(E+D x)^{k-1} e(A x^2 + B x +C)$ is given by\footnote{Here the sum over $j$ is for $0 \leq j \leq \max(k-1,1)$. This convention holds throughout the text.} 
\begin{multline}\label{final fourier equation}  
\left(\frac{i}{2 A } \right)^{k/2} \left(\frac{ D i } {2 \sqrt{\pi} }\right)^{k-1}  \sum_{j}  \left(G -F\left(\frac{\xi+B}{2A} \right) \right)^{1-j} \left(\frac{F  (k-1) }{i \sqrt{-2 \pi A i}} \right)^{j}
\\   \times H_{k-1-j}\left(i \sqrt{-2 \pi A i} \left(\frac{\xi + B}{2 A } - \frac{E }{D} \right) \right) e\left(C -\frac{(\xi+B)^2}{4A} \right).
\end{multline}
\end{lemma}
\begin{proof}
We let $f(x) \coloneqq \left(E- \frac{D B}{2 A }+\frac{Dx}{\sqrt{-2 A i}}\right)^{k-1} e\left(\tfrac{2Ai x^2- B^2 + 4AC }{4A} \right).$
Then we can use the binomial theorem, Lemma \ref{Easy Fouriers} (4) and the identity $H_{n}(x+y) = \sum_{m=0}^{n} \binom{n}{m}H_{m}(x)(2 y )^{n-m}$ to find that $\hat{f}( \xi)$ is equal to
\begin{align*}   & \sum_{n=0}^{k-1} \tbinom{k-1}{n} \left(\tfrac{\sqrt{-8 \pi A i}}{ D i}\left(E-\tfrac{DB}{2A} \right) \right)^{k-1-n} \left(\tfrac{ D i}{\sqrt{-8 \pi A i} } \right)^{k-1} H_{n}(\sqrt{\pi} \xi) e\left(\tfrac{2Ai \xi^2 - B^2 + 4AC }{4A} \right)
\\ = {} & \left(\tfrac{ D i} { \sqrt{-8 \pi A i} }\right)^{k-1}H_{k-1}\left(\sqrt{\pi} \xi + \left(E- \tfrac{D B}{2 A }\right) \tfrac{ \sqrt{-2 \pi A i}}{D i}\right)  e\left(\tfrac{2Ai \xi^2 - B^2 + 4AC }{4A} \right).
\end{align*}
We observe that $ f\left(\frac{-2 A i x - B i}{\sqrt{-2A i}}\right) = (E+Dx)^{k-1} e(A x^2 + B x +C)$. So the Fourier transform of $(E+Dx)^{k-1}e(Ax^{2}+Bx+C)$ can be found by  using Lemma \ref{Easy Fouriers} (1)(3)  to obtain
\[ \left(\tfrac{i}{2 A } \right)^{k/2} \left(\tfrac{ D i } {2 \sqrt{\pi} }\right)^{k-1}H_{k-1}\left(i \sqrt{-2 \pi A i} \left(\tfrac{\xi + B}{2 A } - \tfrac{E }{D} \right) \right)
 e\left(C -\tfrac{(\xi+B)^2}{4A} \right).
\] The stated result then follows using Lemma \ref{Easy Fouriers} (2) and $H'_{n}(x) = 2nH_{n-1}(x)$.
\end{proof}
\subsection{A Theta Function on the Sublattice}
We rewrite our kernel function as a sum of some specific Siegel theta functions defined on the sublattice $K$.
\begin{definition}\label{Hermite theta}
Let $\alpha, \beta \in \mathbb{Z}, h \in K'/K, \mu_{K} \in K \otimes_{\mathbb{Z}} \mathbb{R}$. Then for $ \kappa \geq 0$ we define
\begin{multline*} \xi_{\kappa,h}(\tau, \mu_{K}, \alpha, \beta) {} \coloneqq {}  v^{-\kappa/2}\sum_{\substack{\lambda \in K+rh \\ t(\Delta) \\ Q(\lambda - \beta l' +tl) \equiv \Delta Q(h)(\Delta)}} H_{\kappa} \left(\frac{ \sqrt{ \pi}(\alpha -\beta \overline{\tau} -2|D| v  (\lambda+\beta\mu_{K},\mathfrak{w}^{\perp}))}{ \sqrt{2|\Delta|v Q_{z}(l)}} \right)
\\  \times     \chi_{\Delta}(\lambda - \beta l' +tl) e\left(\frac{-\alpha t} {|\Delta|} \right)e\left(\frac{Q(\lambda + \beta \mu_{K})\tau}{|\Delta|} - \frac{(\lambda + \beta \mu_{K}/2,\alpha\mu_{K})}{|\Delta|}\right),
\end{multline*}
and a $\mathbb{C}[K'/K]$-valued version $\Xi_{\kappa}(\tau,\mu_{K}, \alpha, \beta) \coloneqq\sum_{h \in K'/K}\xi_{\kappa,h}(\tau, \mu_{K}, \alpha, \beta) \mathfrak{e}_{h}.$
\end{definition}
\begin{lemma}\label{Xi-trafo}
For any $(\gamma, \phi_{\gamma}) \in \tilde{\Gamma}, \gamma = \left( \begin{smallmatrix} a&b \\ c&d \end{smallmatrix} \right)$ then
\[ \Xi_{\kappa}(\gamma\tau, \mu_{K} ,a\alpha +b \beta, c \alpha + d\beta) = \phi_{\gamma}(\tau)^{1+2\kappa}\tilde{\rho}_{K}(\gamma,\phi_{\gamma})\Xi_{\kappa}(\tau, \mu_{K}, \alpha, \beta). 
\]
\end{lemma}
\begin{proof} For $\kappa=0,1$, this is just the transformation property of the usual unary theta series, see eg, \cite[Equation~4.5]{BO-Annals}. The general case follows by applying the Maass raising operator, observing
\begin{align*}
 & R_{\kappa-3/2}\left[v^{-(\kappa-2)/2} H_{\kappa-2}(a \sqrt{v})e^{i a^{2}  \tau/2}\right]  \\ {} = {} & \left(\left(\tfrac{\kappa-1}{2v}-a^{2} \right)H_{\kappa-2}(a \sqrt{v})+ \tfrac{(\kappa-2)a}{\sqrt{v}} H_{\kappa-3}(a \sqrt{v}) \right) v^{1-\kappa/2} e^{i a^{2}  \tau/2}  
\\ {} = {} &  \left( (\kappa-1) H_{\kappa-2}(a \sqrt{v})- a \sqrt{v} H_{\kappa-1}(a \sqrt{v}) \right) v^{-\kappa/2} \frac{e^{i a^{2}  \tau/2}}{2}  
 {} = {}   -H_{\kappa}(a \sqrt{v}) v^{-k/2} \frac{e^{i a^{2} \tau/2}}{4}. \qedhere
\end{align*} 
\end{proof}
 
 We note

\begin{lemma}\label{Rewritten Theta}

\begin{multline*} 
\xi_{\kappa, h}(\tau, \mu_{K}, -n, 0) \\
=
  \left( \frac{\Delta}{n} \right)\epsilon_{\Delta}|\Delta|^{1/2} v^{-\kappa/2} \hspace{-.4cm}\sum_{\substack{\lambda \in K +rh \\ Q(\lambda ) \equiv \Delta Q(h)(\Delta)}} H_{\kappa}\left( \frac{\sqrt{\pi}(n-2|\Delta|v(\lambda,\mathfrak{w}^{\perp}))}{\sqrt{2|\Delta|vQ_{z}(l)}}\right)
 e\left(\frac{Q(\lambda)\tau}{|\Delta|} - \frac{(\lambda, n\mu_{K})}{|\Delta|}\right).
\end{multline*}
\end{lemma}
\begin{proof}
Follows easily after observing that if $\lambda \in K +rh, Q(\lambda) \equiv \Delta Q(h) (\Delta)$ then $\chi_{\Delta}(\lambda + tl)= \left(\frac{\Delta}{n} \right)$ so we obtain the stated result by using the Gauss sum in \cite[Equation~4.7]{BO-Annals}.\end{proof} $\epsilon_{\Delta}$ is defined to equal $1$ if $\Delta>0$ or $i$ if $\Delta<0$. $\left(\frac{\Delta}{0}\right) = 0$ if $\Delta \neq 1$ and if $\Delta = 1$ then $\left(\frac{\Delta}{0}\right) = 1$. This means that $\xi_{\kappa, h}(\tau, 0, 0, 0) = 0$ unless $\Delta = 1$. We finally record

\begin{lemma}\label{Raise lower xi} Let $\kappa \in \mathbb{Z}, \kappa \geq 2$ and $\Delta=1$. Then 
\begin{align*}
R_{k-3/2}\Xi_{\kappa-2}(\tau,0,0,0) & = -\frac{1}{4}\Xi_{\kappa}(\tau,0,0,0),
\\ L_{\kappa+1/2}\Xi_{\kappa}(\tau,0,0,0) & = \kappa(\kappa-1)\Xi_{\kappa-2}(\tau,0,0,0).
\end{align*}
\end{lemma}
\begin{proof}
The first claim follows from the proof of Lemma~\ref{Xi-trafo}, while the second from
\begin{align*}
 \frac{\partial}{\partial v}\left( H_{\kappa}(a \sqrt{v}) v^{-\kappa/2} \right) & = -\frac{\kappa}{2 v^{(\kappa+2)/2}} \left(H_{\kappa}(a \sqrt{v})-2a\sqrt{v}H_{\kappa-1}(a \sqrt{v})  \right) = \frac{\kappa(\kappa-1)}{ v^{(\kappa+2)/2}}H_{k-2}(a\sqrt{v}).
\end{align*}
\end{proof}

\subsection{The Poincar\'e Series}
The next few technical lemmas involve rewriting our kernel function in various forms with the aim of writing it as a Poincar\'{e} series in Theorem \ref{Poincare Series}. We will use the following constant \begin{equation}\label{big constant} c_{z,k,j} =\frac{i}{2\sqrt{2|\Delta|}Q_{z}(l)} \left( \frac{q_{z}(l)i \sqrt{|\Delta|}}{2\sqrt{2 \pi Q_{z}(l)}}\right)^{k-1}\left(\frac{(1-k)\sqrt{2|\Delta|Q_{z}(l)}}{\sqrt{\pi}} \right)^{j}.
\end{equation} Next we apply partial Poisson summation, see also \cite[Lemma~5.1]{Borcherds}, \cite[Lemma~4.6]{BO-Annals}, and \cite[Lemma~2.3]{B-Habil}.
\begin{lemma}\label{Big Lemma} We have 
\begin{multline*}
 \theta_{h,k}(\tau,z) {} = {}  \sum_{\lambda \in L/\mathbb{Z}\Delta l +rh} \sum_{\substack{t(\Delta) \\ Q(\lambda+tl) \equiv \Delta Q(h)(\Delta)}}\sum_{d \in \mathbb{Z}} \sum_{j}  \chi_{\Delta}(\lambda + t l) c_{z,k,j}    v^{-(k-1-j)/2}    \left(d+(\lambda,l)\tau \right)^{1-j}  \\  \times  H_{k-1-j} \left(\tfrac{ \sqrt{ \pi}(d + (\lambda,l) \overline{\tau} -2 v  (\lambda,\mathfrak{w}^{\perp}))}{\sqrt{2|\Delta|v Q_{z}(l)}} \right) e\left( -\tfrac{dt}{|\Delta|} \right)  e\left(\tfrac{\tau Q(\lambda_{w^{\perp}}) }{|\Delta|}  - \tfrac{d(\lambda, l_{z^{\perp}} - l_{z})}{2 |\Delta|Q_{z}(l)} - \tfrac{|d + (\lambda, l)\tau|^2}{4|\Delta|ivQ_{z}(l)} \right).\end{multline*}
\end{lemma}
\begin{proof}
We rewrite the sum over $\lambda \in rh +L$ in the definition of $\theta_{h,k}(\tau,z)$  as a sum over $\lambda' + d|\Delta|l$. This is where $\lambda'$ runs over $rh +L/\mathbb{Z}\Delta l$ and $d$ runs over $\mathbb{Z}$. Noting that $\chi_{\Delta}(\lambda + d|\Delta|l) = \chi_{\Delta}(\lambda)$ and $Q(\lambda +d|\Delta|l) \equiv Q(\lambda)(\Delta)$ we then obtain
\begin{equation}\label{g form}
\theta_{h,k}(\tau,z) {} =  v^{3/2}  \sum_{\substack{\lambda \in L/\mathbb{Z}\Delta l +rh \\ Q(\lambda) \equiv \Delta Q(h)(\Delta) }} \chi_{\Delta}(\lambda) \sum_{d \in \mathbb{Z}}g\left(|\Delta|\tau, z, \frac{\lambda}{|\Delta|},k,d\right)
\end{equation}
where
\[ g(\tau, z , \lambda , k, d)  \coloneqq |\Delta|^{k}p_{z}(\lambda +d l)q_{z}(\lambda +dl)^{k-1}e\left(Q(\lambda +dl)u + Q_{z}(\lambda+dl)iv \right). 
\]
We notice
\[ Q(\lambda +dl)u + Q_{z}(\lambda+dl)iv = Ad^2 + Bd + C \]
where $A = Q(l_{z^{\perp}})(\tau - \overline{\tau}) = Q_{z}(l)iv, B = (\lambda, l_{z})\overline{\tau} + (\lambda, l_{z^{\perp}})\tau$ and $C = Q(\lambda_{z})\overline{\tau} + Q(\lambda_{z^{\perp}})\tau = Q(\lambda)u+Q_{z}(\lambda) iv$. We also set $D= q_{z}(l), E= q_{z}(\lambda), F= p_{z}(l)$ and $G= p_{z}(\lambda)$. Then the partial Fourier transform of $ g(\tau, z , \lambda , k, d)$ in $d$ can be found using Lemma \ref{Final Fourier} to obtain $\hat{g}(\tau, z, \lambda,k, d)$.

Using the Poisson summation formula on the variable $d$ we replace $g(|\Delta|\tau, z, \frac{\lambda}{|\Delta|},k,d)$ with $\hat{g}(|\Delta|\tau, z, \frac{\lambda}{|\Delta|},k,d)$ in \eqref{g form}. So then with the help of some simplifying identities given in \cite[Lemma~5.4.1]{Crawford} we obtain
\begin{multline*}
\theta_{h,k}(\tau,z)=   \sum_{\substack{\lambda \in L/\mathbb{Z}\Delta l +rh \\ Q(\lambda) \equiv \Delta Q(h)(\Delta) }} \sum_{d \in \mathbb{Z}} \sum_{j}  \chi_{\Delta}(\lambda)c_{z,k,j} v^{-(k-1-j)/2} \left(d+(\lambda,l)\tau \right)^{1-j}
\\  \times  H_{k-1-j} \left(\tfrac{ \sqrt{ \pi}(d + (\lambda,l) \overline{\tau} -2 v  (\lambda,\mathfrak{w}^{\perp}))}{\sqrt{2|\Delta|v Q_{z}(l)}} \right) e\left(\tfrac{\tau Q(\lambda_{w^{\perp}}) }{|\Delta|}  - \tfrac{d(\lambda, l_{z^{\perp}} - l_{z})}{2 |\Delta|Q_{z}(l)} - \tfrac{|d + (\lambda, l)\tau|^2}{4|\Delta|ivQ_{z}(l)} \right).
\end{multline*}

Finally we rewrite the sum over $\lambda \in L/ \mathbb{Z} Dl + rh$ as a sum over $\lambda' + tl$ where $\lambda'$ and $t$ run over $L/\mathbb{Z}l + rh$ and $\mathbb{Z}/ D{\mathbb{Z}}$ respectively to obtain the stated result.
\end{proof}

We now reformulate \ref{Big Lemma} in terms of the theta function (Definition \ref{Hermite theta}) on the sublattice, see also \cite[Theorem~5.2]{Borcherds}, \cite[Lemma~4.7]{BO-Annals} and \cite[Theorem~2.4]{B-Habil}.

\begin{proposition}\label{Rewritten Lemma}For $h \in L'/L \cong K'/K$,
\[ \theta_{h,k}(\tau,z) = \sum_{c,d \in \mathbb{Z}} \sum_{j}c_{z,k,j}(c\tau + d)^{1-j}
 e\left(- \frac{|c\tau+d|^2}{4|\Delta|iv Q_{z}(l)} \right) \xi_{k-1-j, h}(\tau, \mu_{K}, d, -c).
\]
\end{proposition}

\begin{proof} We use that $L/\mathbb{Z}l + rh \cong K + \mathbb{Z}l' + rh$ to rewrite Lemma \ref{Big Lemma} in terms of $\lambda \in K+rh.$ We do this by making the ``substitution" $\lambda\mapsto \lambda +cl'$, where now $\lambda \in K \otimes \mathbb{R}$ and $c \in \mathbb{Z}$. We have that $(l,l')=1$ and $(\lambda,l)=0$. Using this alongside the identities in Lemma \ref{mu facts} and  inserting the definition of $\xi_{h}(\tau, \mu_{K}, \alpha, \beta, n)$ we obtain the stated result. 
\end{proof}
\begin{theorem}\label{Poincare Series} We have 
\begin{multline*}
\Theta_{k}(\tau,z)  \\=  \frac{1}{2} \sum_{n \geq 1}\sum_{\tilde{\gamma} \in\tilde{\Gamma}_{\infty} \backslash \tilde{\Gamma}} \sum_{j}c_{z,k,j} (-n)^{1-j}
  \left[  e\left(- \tfrac{n^2}{4|\Delta|i \Im(\tau) Q_{z}(l)} \right) \Xi_{k-1-j}(\tau, \mu_{K}, -n, 0) \right] \bigg|_{k-3/2, \tilde{\rho}_{K}} \tilde{\gamma},
\end{multline*}
and if $k \geq 2$ we also have the additional term $c_{z,k,1} \Xi_{k-2}(\tau, 0, 0, 0).$ \end{theorem} \begin{proof} We use Proposition \ref{Rewritten Lemma} to obtain that $\Theta_{k}(\tau,z)$ is equal to \begin{equation*}
 \sum_{n \geq 1}  \sum_{\substack{c,d \in \mathbb{Z} \\ (c,d) =1}} \sum_{j}c_{z,k,j} (-n)^{1-j}(c \tau+d)^{1-j} e\left(- \frac{n^2|d + c\tau|^2}{4|\Delta|iv Q_{z}(l)} \right) \Xi_{k-1-j}(\tau, \mu_{K}, -nd, nc).
\end{equation*} 
plus the stated additional term (when $c=0,d=0,j=1,k \geq 2$)

We know two elements  $ \left(\begin{smallmatrix} a & b \\ c & d \end{smallmatrix}\right),  \left(\begin{smallmatrix} a' & b' \\ c' & d' \end{smallmatrix}\right) \in \Gamma$ are equal in $\Gamma_{\infty} \backslash \Gamma$ if and only if $c=c'$ and $d=d'$. We rewrite the sum over coprime integers as a sum over $\tilde{\gamma} =\left(  \gamma , \phi_{\gamma} \right) \in \tilde{\Gamma}_{\infty} \backslash \tilde{\Gamma}$ where $\gamma = \left(\begin{smallmatrix} a & b \\ c & d \end{smallmatrix}\right)$. This also introduces a factor of $1/2$ due to the two possibilities $\left(  \gamma , \phi_{\gamma} \right)$ and $\left(  \gamma , -\phi_{\gamma} \right)$. We then obtain the stated result using
\[
 \Xi_{k-1-j}(\tau, \mu_{K}, -nd, nc) = \phi_{\gamma}(\tau)^{1 - 2k + 2j} \tilde{\rho}_{K}^{-1}(\gamma,\phi_{\gamma})\Xi_{k-1-j}( \gamma \tau, \mu_{K}, -n, 0). \qedhere
\]
\end{proof}

\subsection{Asymptotics}\label{asymptotics}
We consider the asymptotic behaviour of our kernel function at the cusps. We do this by considering the cusp $l= l_{\infty} = \begin{psmallmatrix} 0 & 1/N \\ 0 & 1 \end{psmallmatrix}$ (so $l'= - l_{0} = \begin{psmallmatrix} 0 & 0 \\ 1 & 0 \end{psmallmatrix}$). We will also associate the upper half plane with an open subset of $K \otimes \mathbb{C}$ by mapping $z' \in \mathbb{H}$ to $ \left(\begin{smallmatrix} 1 & 0 \\ 0 & -1\end{smallmatrix} \right) \otimes z'$. We then have the following identities
\begin{enumerate}
\item $(\lambda, \mu_{K}) = (\lambda,x)$,
\item $(\lambda, \mathfrak{w}^{\perp}) = -\frac{1}{2Ny^{2}}(\lambda,y)$,
\item $Q_{z}(l) = \frac{1}{2Ny^{2}}$,
\item $c_{z,k,j} = \frac{i N y^{2}}{\sqrt{2|\Delta|}} \left( \frac{i y \sqrt{|\Delta|}}{2\sqrt{2  \pi}}\right)^{k-1}\left(\frac{(1-k)\sqrt{|\Delta|}}{y\sqrt{ N \pi}} \right)^{j}$.
\end{enumerate}
We will denote $c_{k,j} \coloneqq y^{k-1-j}c_{z,k,j}$ and for any function $f: \mathbb{H} \rightarrow \mathbb{C}[L'/L]$ with components $f_{h}$ we will denote
$f_{W_{m}^{N}} \coloneqq \sum_{h \in L'/L} f_{W_{m}^{N}.h} \mathfrak{e}_{h}.	
$
\begin{proposition}\label{cusp growth of theta}
Let $N$ be square-free. Let $s =W_{m}^{N} \infty \in \mathbb{P}^{1}(\mathbb{Q})$ be a cusp of $\Gamma_{0}(N)$. Then there is a constant $ C>0$ such that as $y \rightarrow \infty$ we have
\[ \left(\Theta_{k}(\tau)\big|_{2-2k}W_{m}^{N}\right)(z)  = y^{k} c_{k,1}\Xi_{k-2, W_{m}^{N}}(\tau,0,0,0) + \mathcal{O}(e^{-Cy^{2}}) \]
unless $k=1$, in which case $\left(\Theta_{1}(\tau)\big|_{2-2k}W_{m}^{N}\right)(z) = \mathcal{O}(e^{-Cy^{2}}).$
\end{proposition}

\begin{proof}
Using Proposition \ref{Rewritten Lemma} and the above identities we see that $\theta_{h,k}(\tau,z)$ decays exponentially as $y \rightarrow \infty$ (uniformly in $x$) except the case when $c=0,d=0$. In that case we observe $\theta_{h,k}(\tau,z)$ vanishes unless $k \geq 2$ and $j=1$. In the remaining cases, Theorem \ref{Poincare Series} tells us we have $c_{z,k,1}\xi_{k-2,h}(\tau,0,0,0)$ left to consider. The explicit identities given above also tell us that $\xi_{k-2,h}(\tau,0,0,0)$ does not depend on $y$ and so 
\[ \Theta_{k}(\tau,z) = c_{k,1}y^{k}\Xi_{k-2}(\tau,0,0,0) + \mathcal{O}(e^{-Cy^{2}})
\]
as $y \rightarrow \infty$. The statement about the other cusps follows from Proposition~\ref{atkin lehner transformation}.
\end{proof}

\section{The Fourier Expansion}\label{The Fourier Expanison Section}
In this section we compute the Fourier expansion of $\Phi_{\Delta,r,k}(z,f)$ and conclude that $\Phi_{\Delta,r,k}(z,f)$ is indeed a locally harmonic weak Maass form as in Definition \ref{locally harmonic weak maass form}. We need to do some groundwork first. 

\begin{lemma}\label{additional integral}
Let $f \in H_{3/2-k,\overline{\rho}}$ with expansion \eqref{Simple Expansion}. Let $k \geq 2, \Delta=1$. Then
\[  \int_{\tau \in \mathcal{F}}^{\mathrm{reg}} \left\langle f(\tau), \overline{\Xi_{k-2}(\tau,0,0,0)} \right\rangle \frac{dudv}{v^{2}} =   \sum_{h \in K'/K} \sum_{\lambda \in K+h}\frac{c^{+}(-Q(\lambda),h)}{k(k-1)}\left(\frac{ -2\sqrt{2 \pi} (\lambda,\mathfrak{w}^{\perp})}{\sqrt{Q_{z}(l)}}\right)^{k}.
\]
\end{lemma}
\begin{proof}
We write $g = \Xi_{k-2}(\tau,0,0,0)$. We first show $(R_{k-3/2}(g), \xi_{3/2-k}(f))^{\mathrm{reg}}_{1/2-k,L} = 0$. Indeed, we have using \cite[Lemma~4.2]{{B-Habil}} (correcting a sign error):
\begin{multline*} (R_{-\kappa}(g),\xi_{\kappa}(f))^{\mathrm{reg}}_{2-\kappa, L}
=  \lim_{t \rightarrow \infty}\int_{\tau \in \mathcal{F}_{t}} \left\langle R_{-\kappa}(g), \xi_{\kappa}(f) \right\rangle v^{-\kappa} du dv \\ =  -\lim_{t \rightarrow \infty} \int_{\tau \in \mathcal{F}_{t}} \left\langle g, L_{2-\kappa}\xi_{\kappa}(f) \right\rangle v^{-\kappa} \frac{du dv}{v^{2}}  + \lim_{t \rightarrow \infty}\int_{-1/2}^{1/2} \left[\left\langle g,\xi_{\kappa}(f) \right\rangle v^{-\kappa} \right]_{v=t} du.
\end{multline*}
We observe the second term disappears as  $g$ is bounded and $\xi_{\kappa}(f)$ is exponentially decaying. The first term also disappears as $f$ is harmonic so
$L_{2-\kappa}\xi_{\kappa}(f) = v^{\kappa}\overline{\xi_{2-\kappa}\xi_{\kappa}(f)} = -v^{\kappa}\overline{\Delta_{\kappa}(f)} =0.$ Using \cite[Lemma~4.2]{{B-Habil}} in the other direction gives
\begin{multline*}\label{lowering raising}
 \lim_{t \rightarrow \infty}\int_{\tau \in \mathcal{F}_{t}} \left\langle R_{-\kappa}(g), \xi_{\kappa}(f) \right\rangle v^{-\kappa} du dv  = \lim_{t \rightarrow \infty}\overline{\int_{\tau \in \mathcal{F}_{t}} \left\langle R_{-\kappa}(v^{\kappa} \overline{f}), R_{-\kappa}(g) \right\rangle v^{-\kappa} du dv} 
\\ = - \lim_{t \rightarrow \infty}\int_{\tau \in \mathcal{F}_{t}} \left\langle L_{2-\kappa}R_{-\kappa}(g),  \overline{f}   \right\rangle  \frac{du dv}{v^2} + \lim_{t \rightarrow \infty}\int_{-1/2}^{1/2} \left[ \left\langle f, \overline{R_{-\kappa}(g)}  \right\rangle \right]_{v=t} du. 
\end{multline*} 
Using Lemma \ref{Raise lower xi} we obtain 
\[
 \int_{\tau \in \mathcal{F}}^{\mathrm{reg}} \left\langle f(\tau), \overline{g(\tau)}  \right\rangle \frac{du dv}{v^2} =  -\frac{4}{k(k-1)}\lim_{t \rightarrow \infty}\int_{-1/2}^{1/2} \left\langle f(u+it), \overline{R_{k-3/2}(g(u+it))}  \right\rangle du.
\]
The integral over $u$ picks out the $0$-th Fourier coefficient and $f^{-}(u+it)$ decays exponentially as $t \rightarrow \infty$ so it remains to consider
\begin{align*} {} = {} & \frac{1}{k(k-1)}\lim_{t \rightarrow \infty} \sum_{h \in K'/K} \sum_{\lambda \in K+h}c(-Q(\lambda),h,t)t^{-k/2}H_{k}\left(\frac{ -\sqrt{2 \pi t} (\lambda,\mathfrak{w}^{\perp})}{\sqrt{Q_{z}(l)}} \right).
\intertext{We have a finite set of terms $c^{+}(-Q(\lambda),h)$. Using $H_{n}(x) = n! \sum_{m=0}^{ \lfloor n/2 \rfloor} \frac{(-1)^{m}(2x)^{n-2m}}{m!(n-2m)!}$ we obtain the result after taking the limit}
 {} = {} & \lim_{t \rightarrow \infty} \sum_{h \in K'/K} \sum_{\lambda \in K+h}\frac{t^{-k/2} c^{+}(-Q(\lambda),h)k!}{k(k-1)}\sum_{n=0}^{ \lfloor k/2 \rfloor} \tfrac{(-1)^{n}}{n!(k-2n)!} \left(\tfrac{ -2\sqrt{2 \pi t} (\lambda,\mathfrak{w}^{\perp})}{\sqrt{Q_{z}(l)}}\right)^{k-2n}. \qedhere
\end{align*} 
\end{proof}

We need a few more integrals.

\begin{lemma}\label{Integral 0}
Let $n \in \mathbb{Z}, n \geq 0$ and let $r \in \mathbb{Z}, r \geq n$. Then
\begin{equation}\label{easy integral} \int_{t=0}^{\infty} t^{r}  
 H_{n}\left(t \right)e^{- t^{2} } dt  = \frac{r!}{2(r-n)!}\Gamma\left(\frac{r-n+1}{2}\right). \end{equation}
\end{lemma}
\begin{proof}
This follows from $\frac{\partial H_{n}\left(t \right)e^{- t^{2}}}{\partial t}= -H_{n+1}(t)e^{- t^{2}}$ and integration by parts.
\end{proof}

 \begin{lemma}\label{Integral 1}
Let $\kappa \in \mathbb{Z}, \kappa \geq 0$ and let $\alpha, \beta \in \mathbb{R}, \alpha >0$. Then
\[
\int_{v=0}^{\infty}\sum_{j} v^{-\kappa/2}  
 H_{\kappa-j}\left(-\frac{\alpha}{\sqrt{v}}+ \beta\sqrt{v} \right)\left(\frac{\kappa \sqrt{v}}{\alpha} \right)^{j} e^{- \alpha^{2}/v} \frac{dv}{v^{2}} = \frac{e^{-2 \alpha \beta}\Gamma(\kappa+1,-2 \alpha \beta)}{(-\alpha)^{\kappa+2}}.
\]
\end{lemma}
\begin{proof} Substituting $v = \alpha^{2}/t^{2}$ and $H_{n}(x+y) = \sum_{m=0}^{n} \binom{n}{m}H_{m}(x)(2 y )^{n-m}$ we obtain
\[ \frac{2(2 \beta)^{\kappa}}{\alpha^{2}} \sum_{j}\sum_{n=0}^{\kappa-j} \binom{\kappa-j}{n} \left(\frac{\kappa}{2 \alpha \beta} \right)^{j} (-2\alpha\beta)^{-n} \int_{t=0}^{\infty} t^{n+1}  
 H_{n}\left(t \right)e^{- t^{2} } dt.  
\]
Applying Lemma \ref{Integral 0} for $r=n+1$ this becomes
\[ 
\frac{2(2 \beta)^{\kappa}}{2 \alpha^{2}} \sum_{j}\sum_{n=0}^{\kappa-j} \binom{\kappa-j}{n} \left(\frac{\kappa}{2 \alpha \beta} \right)^{j} (-2\alpha\beta)^{-n} (n+1)!, 
\]
which we simplify and using the identity from \cite[Section~11.1.9]{Handbook} to get
\begin{align} = &  (-\alpha)^{-\kappa-2} \kappa!\sum_{n=0}^{\kappa} \frac{(-2\alpha\beta)^{n}}{n!} = \frac{e^{-2 \alpha \beta}\Gamma(\kappa+1,-2 \alpha \beta)}{(-\alpha)^{\kappa+2}}. \qedhere
\end{align} \end{proof}

\begin{lemma}\label{Integral 2}
Let $\kappa \in \mathbb{Z}, \kappa \geq 0$ and let $\alpha, \beta \in \mathbb{R}, \alpha >0, \beta \neq 0$. Then
\begin{align*}
& \int_{v=0}^{\infty}\sum_{j} v^{-\kappa/2}  
 H_{\kappa-j}\left(-\frac{\alpha}{\sqrt{v}}+ \beta\sqrt{v} \right)\left(\frac{\kappa \sqrt{v}}{\alpha} \right)^{j} e^{- \alpha^{2}/v}\Gamma\left(\kappa+1/2, \beta^{2} v \right) \frac{dv}{v^{2}}
\\  {} = {} & \begin{dcases}
	 \left(-1 \right)^{\kappa} \frac{(2\kappa)!\sqrt{\pi}}{4^{\kappa}\alpha^{\kappa+2}}e^{-2 \alpha \beta} & \textrm{if }  \beta >0,
	\\ \left(-1 \right)^{\kappa} \frac{\sqrt{\pi}}{4^{\kappa}\alpha^{\kappa+2}}e^{-2 \alpha \beta}\Gamma(2\kappa+1, - 4 \alpha \beta) & \textrm{if }  \beta <0.
\end{dcases}
\end{align*}
\end{lemma}
\begin{proof} 
In the integral we write $\alpha=a y$ and write for the Integral $I(y)$. Then $I$ satisfies the second order differential equation
\[
yI''(y) -2\kappa I'(y) - (4\kappa a\beta+4a^2\beta^2y)I(y) = 0.
\]
We easily check that $e^{-2a\beta y}$ and $e^{-2a\beta y}\Gamma(2\kappa+1,-4a\beta y)$ are linear independent solutions. The result can then be obtained by using Lemma \ref{Integral 1} as part of some (fairly tedious) calculations determining the asymptotics of the integral as $y \to \infty$. 
\end{proof}

To state the Fourier expansion we need some additional notation. 

For $A \subset \mathbb{R}$ we define the indicator function $I_{A}(x)$ to equal $1/2$ if $x \in A$ and to equal $0$ otherwise. We will denote the fractional part as $\langle x\rangle \coloneqq x - \lfloor x \rfloor$. Recall that the Bernoulli polynomials are defined by the generating function $\frac{t e^{xt}}{e^{t}-1} = \sum_{n=0}^{\infty}B_{n}(x) \frac{t^{n}}{n!}$ while the periodic Bernoulli polynomials are  given by $ \mathbb{B}_{n}(x) \coloneqq -n!\sum_{m \neq 0}\frac{e(mx)}{(2 \pi i m)^n}$, 
where we set $\mathbb{B}_{0}(x)=1$. We let  $L\left(s, \left(\frac{\Delta}{\cdot}\right)\right) = \sum_{n \geq 1} \left(\frac{\Delta}{n}\right)\frac{1}{n^{s}}$ be the Dirichlet L-function. For $\kappa \in \mathbb{Z}, s \in \mathbb{C}, |s| <1$ we denote the polylogarithm as $\textrm{Li}_{\kappa}(s)$ and introduce a shifted incomplete polylogarithm where for $b \in \mathbb{R}_+$ we set 
\[
 \mathrm{Li}_{\kappa,r}(b,s) \coloneqq \sum_{n=1}^{\infty} \frac{s^{n}}{n^{\kappa+r}}\frac{\Gamma\left(\kappa, nb \right)}{\Gamma(\kappa)},
 \]
 which is a finite sum of polylogarithms. Finally we will use the following constants 
\begin{equation*}
 C_{1} \coloneqq \tfrac{\epsilon_{\Delta}|\Delta|\sqrt{2}}{i \pi } \left(\tfrac{|\Delta|}{i \pi 2\sqrt{2N}} \right)^{k-1}, {} C_{2} \coloneqq \tfrac{2\sqrt{2}\epsilon_{\Delta} \sqrt{\Delta}}{k}\left( \tfrac{\Delta}{ \sqrt{2N}}\right)^{k-1}, {} C_{3} \coloneqq \tfrac{ \sqrt{2} \epsilon_{\Delta}\sqrt{\Delta} (2k-2)!}{i \sqrt{\pi}}\left( \tfrac{\Delta}{8\pi i\sqrt{2N }}\right)^{k-1}. 
\end{equation*}

We can now state the main result of this chapter. The proof is similar in nature to those found in \cite[Sections~7 and 14]{Borcherds}, \cite[Chapters~2 and 3]{B-Habil} and \cite[Theorem~5.3]{BO-Annals}.

\begin{theorem}\label{The Fourier Expansion}
Let $f \in H_{3/2-k,\overline{\rho}}$ with expansion \eqref{Simple Expansion}. If $n_{0} <0$ then $\Phi_{\Delta,r,k}(z,f) \equiv 0$. If $n_{0} \geq 0$, let $z=x+iy \in \mathbb{H}$ where $y > \sqrt{-|\Delta|n_{0}/N}$. Then at the cusp $l_{\infty}$
\begin{subequations}\label{big expansion} 
\begin{align}
&\Phi_{\Delta,r,k}(z,f) = \notag \\
&    \qquad   C_{1}  c^{+}(0,0)L\left(k, \left(\tfrac{\Delta}{\cdot}\right)\right)\label{constant term, expansion}
\\  & \quad -  C_{2}  \sum_{m \geq 1}\sum_{b(\Delta)} \left( \frac{\Delta}{b}\right)c^{+}\left(-\frac{|\Delta|m^{2}}{4N},\frac{r m}{2N} \right) \label{bernoulli b}
 \left[ B_{k}\left( \left\langle mx + b/\Delta \right\rangle +imy \right) +  \frac{kI_{\mathbb{Z}}(mx+b/\Delta)}{(imy)^{1-k}}\right] \\
& \quad +  C_{3} \sum_{m \geq 1} \sum_{b(\Delta)}\left( \frac{\Delta}{b}\right) c^{-}\left(-\frac{|\Delta|m^{2}}{4N},\frac{rm}{2N}  \right) \label{c- term,a} \\
&\qquad \qquad \quad \times  \left[ \mathrm{Li}_{k} \left(e( m z+b/\Delta)\right)   + (-1)^{k}\mathrm{sgn}(\Delta) \mathrm{Li}_{2k-1,1-k}\left(4 \pi m y, e(-(m z-b/ \Delta))\right) \right]. \label{c- term,b}
\end{align}
\end{subequations}
In the case $k=1$ square bracketed part of \eqref{bernoulli b} is replaced with $\mathbb{B}_{1}(mx+b/\Delta)$. Note that the constant term \eqref{constant term, expansion} vanishes if $k$ is odd and $\Delta >0$ or if $k$ is even and $\Delta <0$.
\end{theorem}

\begin{proof} Inserting \eqref{Poincare Series} into Definition \ref{The Lift} and using the usual unfolding trick, we have
\begin{align}
&\Phi_{\Delta, r, k}(z,f)  \notag\\
&=   c_{z,k,1} \int_{\tau \in \mathcal{F}}^{\mathrm{reg}} \left\langle   f(\tau) , \overline{\Xi_{k-2}(\tau, 0, 0, 0)}\right\rangle \frac{du dv}{v^2} \label{Term11}
\\ 
 & \quad + 2 \int_{v=0}^{\infty} \int_{u=-\frac12}^{\frac12} \sum_{n \geq 1} \sum_{j}c_{z,k,j}
   (-n)^{1-j} e\left(- \tfrac{n^2}{4|\Delta|i v Q_{z}(l)} \right)  \left\langle  f(\tau) , \overline{\Xi_{k-1-j}(\tau, \mu_{K}, -n, 0)} \right\rangle \frac{du dv}{v^2}. \label{Term22}
\end{align}
We consider \eqref{Term11} first. This disappears unless $k \geq 2$ and $\Delta=1$ (see Theorem \ref{Poincare Series} and Lemma \ref{Rewritten Theta}). If $k \geq 2, \Delta=1$ we use Lemma \ref{additional integral}. This term can then be simplified using the identities in Section~\ref{asymptotics}, identifying $K'$ with $\frac{m}{2N}$ for $m \in \mathbb{Z}$, and observing that $c^{+}\left(-Q(\lambda),\lambda \right) = (-1)^{k}c^{+}\left(-Q(\lambda),-\lambda \right)$. So the sum over $\lambda < 0$ is the same as the sum over $\lambda > 0$. Putting all this together we obtain 
\begin{equation}\label{additional piece untwisted}
 \frac{ 2\sqrt{2} y}{i k}\left( \frac{y i}{\sqrt{2N}} \right)^{k-1} \sum_{m \geq 1}m^{k}c^{+}\left(-\frac{m^{2}}{4N},\frac{m}{2N}\right).
\end{equation} 
For the second term \eqref{Term22}, we insert the expansions given in \eqref{Simple Expansion} and Lemma \ref{Rewritten Theta} and carry out the integration over $u$.  We notice that taking a sum over $\lambda \in K, h \in K'/K$ such that $Q(\lambda) \equiv \Delta Q(h) (\Delta)$ and $\lambda \equiv rh ( K )$ is equivalent to taking a sum over $\lambda' \in K'$ with $\lambda = \Delta \lambda'$ and $r \lambda' \equiv h ( K)$. We obtain
\begin{multline}\label{absolute convergence integral}
 2 \epsilon_{\Delta}|\Delta|^{1/2} \int_{v=0}^{\infty}  \sum_{j}\sum_{\lambda \in K'}\sum_{n \geq 1} \left( \frac{\Delta}{n}\right) c_{z,k,j}(-n)^{1-j}  v^{-(k-1-j)/2} c\left(-|\Delta|Q(\lambda),r \lambda,v \right) \\  \times 
 H_{k-1-j}\left( \frac{\sqrt{\pi}(-n-2 \Delta v( \lambda,\mathfrak{w}^{\perp}))}{\sqrt{2 |\Delta| v Q_{z}(l)}}\right) e\left( \mathrm{sgn}(\Delta)n(\lambda, \mu_{K})\right) e\left(- \frac{n^2}{4|\Delta|i v Q_{z}(l)} \right) \frac{dv}{v^{2}}.
\end{multline}
We now split the Fourier coefficients $c(m,h,v) = c^{+}(m,h) + c^{-}(m,h)\Gamma(k-1/2,-4\pi m v)$. 

{\textbf{The $c^{+}$ Terms:}} Since $K$ is positive definite $-|\Delta|Q(\lambda)$ is negative, and so we will have only finitely many nonzero $c^{+}(-|\Delta|Q(\lambda),r \lambda)$. We set
\begin{equation}\label{alpha beta} 
\alpha: = \frac{n \sqrt{\pi}}{\sqrt{2|\Delta|Q_{z}(l)}} > 0 \quad \beta\coloneqq -\frac{\sqrt{2 \pi }\Delta(\lambda,\mathfrak{w}^{\perp})}{\sqrt{|\Delta| Q_{z}(l)}}
\end{equation}
In this case the integral and the sum over $j$ in \eqref{absolute convergence integral} is exactly as in Lemma \ref{Integral 1} (for $\kappa=k-1$), and we are left with
\begin{multline}\label{Before resuming}
 -2 \epsilon_{\Delta}|\Delta|^{1/2} \left(-\frac{\sqrt{2|\Delta|Q_{z}(l)}}{ \sqrt{\pi}}\right)^{k+1} c_{z,k,0} \sum_{\lambda \in K'} \sum_{n \geq 1} \left( \frac{\Delta}{n}\right) \frac{ c^{+}\left(-|\Delta|Q(\lambda),r \lambda \right)}{n^{k}}  \\  \times 
 e\left( \mathrm{sgn}(\Delta)n\left( (\lambda, \mu_{K})-\frac{i  (\lambda, \mathfrak{w}^{\perp})}{Q_{z}(l)}\right)\right) \Gamma\left(k,\frac{\mathrm{sgn}(\Delta) 2 \pi  n(\lambda, \mathfrak{w}^{\perp})}{Q_{z}(l)}\right).
\end{multline}
We then observe that if we change the sign of $\lambda$ and $n$ and  the terms in the sum remain unchanged. This means we can replace the sums over $n \geq 1$ and $\lambda \in K'$ with sums over $n \in \mathbb{Z}, n \neq 0$ and $\lambda \in K', \lambda > 0$ respectively. The case when $\lambda =0$ gives the constant term \ref{constant term, expansion}. 
We observe \ref{constant term, expansion} disappears if $\frac{1}{2}\left( 1-2k-\mathrm{sgn}(\Delta)\right)$ is odd as then $c^{+}(0,0)=0$. Returning to \eqref{Before resuming} we use \cite[Equation~4.7]{BO-Annals}, the identifications in Section \ref{asymptotics}, and identifying $K' = \frac{1}{2m}\mathbb{Z}$ to obtain 
\begin{multline*}
 \frac{ \sqrt{2 |\Delta|}}{i \pi}\left( \frac{-i|\Delta|}{2 \pi \sqrt{2N}}\right)^{k-1}  \sum_{m \geq 1}\sum_{n \in \mathbb{Z} \backslash \left\lbrace  0 \right\rbrace }\sum_{b(\Delta)} \left( \frac{\Delta}{b}\right) c^{+}\left(-\frac{|\Delta|m^{2}}{4N},\frac{r m}{2N} \right) \\  \times 
n^{-k} e\left( \mathrm{sgn}(\Delta)n\left(mz +\frac{b}{\Delta} \right)\right) \Gamma\left(k,-\mathrm{sgn}(\Delta) 2 \pi n m y\right).
\end{multline*}
Using the decomposition of the incomplete gamma function given in \cite[Section~11.1.9]{Handbook} and properties of the 
Bernoulli polynomials we obtain
\[
 C_{2}  \sum_{m \geq 1}\sum_{b(\Delta)}\sum_{s=1}^{k}\binom{k}{s}\left( \frac{\Delta}{b}\right) c^{+}\left(-\frac{|\Delta|m^{2}}{4N},\frac{r m}{2N} \right) 
 \mathbb{B}_{s}\left( mx +\frac{b}{\Delta} \right)\left(i m y \right)^{k-s}.
\]
This gives the stated result when $k=1$. In the case $k \geq 2$ we know $\sum_{b (\Delta)}\left(\frac{\Delta}{b}\right)$ vanishes if $\Delta \neq 1$ and $\mathbb{B}_{0}(x)=1$ so we can combine \eqref{additional piece untwisted} into this sum (we now sum over $s \geq 0$) to get
\begin{equation}\label{sum of bernoullis}
C_{2} \sum_{m \geq 1}\sum_{b(\Delta)}\sum_{s=0}^{k}\binom{k}{s}\left( \frac{\Delta}{b}\right) c^{+}\left(-\frac{|\Delta|m^{2}}{4N},\frac{r m}{2N} \right) 
 \mathbb{B}_{s}\left( mx +\frac{b}{\Delta} \right)\left(i m y \right)^{k-s},
\end{equation}
from which the form \ref{bernoulli b} stated in the theorem follows using standard properties of the Bernoulli polynomials.

{\textbf{The $c^{-}$ Terms:}} We now have infinitely many nonzero $c^{-}(-|\Delta|Q(\lambda), r \lambda)$. With $\alpha$, $\beta$ as in \eqref{alpha beta} we set $\tilde{\alpha} = \alpha/y>0$ and $\tilde{\beta}=\beta/\mathrm{sgn}(\Delta)m > 0$ for $m \in \mathbb{Z}$. We then obtain
\begin{multline}\label{c- differential}
 -2 \epsilon_{\Delta}|\Delta|^{1/2}c_{k,0}  \sum_{j}\sum_{m \in \mathbb{Z}\backslash \left\lbrace 0 \right\rbrace }\sum_{n \geq 1} \left( \frac{\Delta}{n}\right) c^{-}\left(-\frac{|\Delta|m^{2}}{4N},\frac{rm}{2N}  \right) e^{\mathrm{sgn}(\Delta)2\tilde{\alpha}m\tilde{\beta} i x} n   y^{k+1} \\ \times  \int_{v=0}^{\infty}  \left(\tfrac{(k-1) \sqrt{v}}{\tilde{\alpha} y} \right)^{j}  v^{-\frac{(k-1)}{2}} \Gamma\left(k-1/2,(m\tilde{\beta})^{2} v \right)
 H_{k-1-j}\left( -\frac{\tilde{\alpha} y}{\sqrt{v}}+ \mathrm{sgn}(\Delta)m\tilde{\beta}\sqrt{v}\right) e^{- \frac{(ay)^{2}}{v}} \frac{dv}{v^{2}}.
\end{multline}
We had this integral in $v$ (and sum over $j$) in Lemma \ref{Integral 2}. Carefully considering the cases when $m$ and $\mathrm{sgn}(\Delta)$ are positive and negative we can now switch to a sum over $m \geq 1$ (also remembering
$c^{-}(m,h) = (-1)^{k}\mathrm{sgn}(\Delta)c^{-}(m,-h)$). We obtain
\begin{multline*}
  \frac{ \sqrt{2} \epsilon_{\Delta}\Delta }{i \sqrt{\pi}}\left( \frac{ \Delta}{8\pi i \sqrt{2N }}\right)^{k-1}  \sum_{m \geq 1}\sum_{n \geq 1} \left( \frac{\Delta}{n}\right) c^{-}\left(-\frac{|\Delta|m^{2}}{4N},\frac{rm}{2N}  \right) \\  \times  n^{-k}\left[\mathrm{sgn}(\Delta)  (2k-2)! e(n m z)   + (-1)^{k}e(-n m z) \Gamma(2k-1,4 \pi n m y)\right] .
\end{multline*}
Since $\sum_{b(\Delta)}\left(\frac{\Delta}{b}\right)e\left( \frac{nb}{\Delta}\right) = \mathrm{sgn}(\Delta)\left( \frac{\Delta}{n}\right)\sqrt{\Delta}$ we finally arrive at
\begin{multline}\label{easy form c-}
  \frac{ \sqrt{2} \epsilon_{\Delta}\sqrt{\Delta} }{i \sqrt{\pi}}\left( \frac{ \Delta}{8\pi i \sqrt{2N }}\right)^{k-1}  \sum_{m \geq 1}\sum_{n \geq 1} \sum_{b(\Delta)}\left( \frac{\Delta}{b}\right) c^{-}\left(-\frac{|\Delta|m^{2}}{4N},\frac{rm}{2N}  \right)n^{-k} \\  \times  \left[ (2k-2)! e(n (m z+b/\Delta))   + (-1)^{k}\mathrm{sgn}(\Delta) e(-n (m z-b/\Delta)) \Gamma(2k-1,4 \pi n m y)\right] .
\end{multline}
We can reformulate this more compactly (and remove an infinite sum) by writing this in terms of polylogarithm to obtain the stated result. 
\end{proof}

\begin{remark}
While the expansion a priori only holds when $y>\sqrt{2|\Delta|n_{0}/N}$ the expansion given in \eqref{big expansion} actually converges absolutely for all $z \in \mathbb{H}$.  

We can find expansions in the bounded Weyl chambers as well. Theorem \ref{Singularities} told us that for a point $z_{0} \in \mathbb{H}$ there exists an open neighbourhood $U \subset \mathbb{H}$ so that subtracting some suitable polynomials allowed $\Phi_{\Delta,r,k}(z,f)$ to be continued to a smooth function on $U$. Using the real analyticity of $\Phi_{\Delta,r,k}(z,f)$ and Theorem \ref{wall cross theorem} it is clear the expansion in a bounded Weyl chamber is given by the expansion in \eqref{big expansion} with the addition of the following term
\[ 2\sqrt{2|\Delta|} \sum_{h \in L'/L}  \sum_{\substack{m \in \mathbb{Z} - \textrm{sgn}(\Delta)Q(h) \\ m < 0 }}c^{+}\left(m,h \right) \sum_{W_{ij} \in W} \sum_{\substack{\lambda \in L_{-d\Delta,rh} \\ \lambda \not\perp l_{\infty} \\ \lambda \perp W_{ij} \\ (\lambda,W_{i}) <0}}\chi_{D}(\lambda) q_{z}(\lambda)^{k-1}
\] where $W$ is any set of (semi-circular) walls crossed to reach this Weyl chamber from the unbounded Weyl chamber lying directly above it.
\end{remark}

\begin{proposition}\label{asymptotic of lift}
Let $f \in H_{3/2-k, \overline{\rho}}$ and let $k \geq 2$ . Then
\[ \Phi_{\Delta,r,k}(z,f) = \mathcal{O}(y^{k})
\]
as $y \rightarrow \infty$, uniformly in $x$. If $k=1$ then $\lim_{y \rightarrow \infty}\Phi_{\Delta,r,1}(z,f)$ exists. 
\end{proposition}
\begin{proof}
We use Theorem \ref{The Fourier Expansion}. The first part \eqref{constant term, expansion} was just a constant. The second part  \eqref{bernoulli b} consisted of a finite sum over $m \geq 1$ and we note that $\mathbb{B}_{s}(mx+b/D)$ is bounded for any $m,x$ and it is clear that this part grows $\mathcal{O}(y^{k})$. We now consider the third part as in \eqref{easy form c-}. We will bound the absolute value and using again \cite[Section~11.1.9]{Handbook} it then suffices to consider (up to constants):
\[ \left\vert \sum_{m \geq 1}c^{-}\left( -\frac{|\Delta |m^{2}}{4N}, \frac{rm}{2N} \right) \sum_{n \geq 1}\sum_{s=0}^{2k-2}\frac{e^{-2 \pi m n y}}{n^{k}}( n my )^{s} \right\vert. \]
Using the asymptotic behaviour of $c^{-}(n,h$) this is majorized by 
\[
 C \cdot \sum_{s=0}^{2k-2}y^{s}\sum_{m \geq 1}m^{1+s}e^{- \pi m y} \sum_{n \geq 1} n^{s}e^{- \pi n y} \leq C \cdot \sum_{s=0}^{2k-2}y^{s} \mathrm{Li}_{-s-1}(e^{-\pi y})  \mathrm{Li}_{-s}(e^{-\pi y}),
\]
which decays exponentially as $y \rightarrow \infty$.
\end{proof}

Putting everything together gives 

\begin{theorem}\label{lift is a local maasd form theorem}
Let $f \in H_{3/2-k, \overline{\rho}}$ and let $N$ be square-free. Then $\Phi_{\Delta,r,k}(z,f)$ is a locally harmonic weak Maass form of weight $2-2k$ for $\Gamma_{0}(N)$ with exceptional set $Z'_{\Delta,r}(f)$.
\end{theorem}
\begin{proof}
We  recall $Z_{\Delta,r}(f)$ was a nowhere dense $\Gamma_{0}(N)$-invariant set. Items (1),(2) and (3) in Definition \ref{locally harmonic weak maass form} are clear from Theorem \ref{Singularities}, Theorem \ref{Locally Harmonic} and Theorem \ref{Singularities}. It remains to check the cusp condition. But this follows easily considering the action of the Atkin-Lehner involutions on the integral kernel. 
 \end{proof}

\section{The Shimura Lift}
In this section we consider the relationship of our lift with the Shimura correspondence. We first link the two lifts via the differential operator $\xi_{2-2k}$ and then using this we obtain new proofs of the properties of the Shimura lift.

\begin{definition}\label{Definition shimura lift}
For $g \in S_{k+1/2, \rho}$ we define the Shimura lift as
\[ \Phi_{\Delta,r,k}^{*}(z,g) {} \coloneqq {} \left( g(\tau),\Theta_{\Delta,r,k}^{*}(\tau,z) \right)_{k+1/2,\rho}.
\]
\end{definition}
The next key connection goes back to \cite{BF-Duke} and \cite{FM-coeff} and here is adapted from \cite[Lemma~3.3]{BKV}.
 
 \begin{lemma}\label{Link Thetas} We have 
\[ \xi_{k+1/2, \tau}( \Theta_{\Delta,r,k}^{*}(\tau,z)) = -\frac{1}{2}\xi_{2-2k,z}\left(v^{k-3/2}\Theta_{\Delta,r,k}(\tau,z)\right).
\]
\end{lemma}

The next result is adapted from \cite[Lemma~3.4]{BKV}; results of this type go back to \cite{BF-Duke}. 

\begin{theorem}\label{Link}
Let $f \in H_{3/2-k,\overline{\rho}}$ and $z \in \mathbb{H} \backslash Z'_{\Delta,r}(f)$. Then
\[ \Phi_{\Delta,r,k}^{*}(z,\xi_{3/2-k}(f)) = \frac{1}{2}\xi_{2-2k,z}(\Phi_{\Delta,r,k}(z,f)). \]
\end{theorem}
\begin{proof}
Using Stokes' theorem we see the left hand side is equal to
\begin{align*}
 -\lim_{t \rightarrow \infty}\int_{\tau \in \mathcal{F}_{t}}\left\langle \overline{f(\tau)}, L_{k+1/2,\tau}\left(\Theta_{\Delta,r,k}^{*}(\tau,z)\right) \right\rangle  \frac{dudv}{v^2} + \lim_{t \rightarrow \infty}\int_{-1/2}^{1/2}\left[\left\langle \overline{f(\tau)},\Theta_{\Delta,r,k}^{*}(\tau,z)  \right\rangle   \right]_{v=t}du.
\end{align*}
Using Lemma \ref{Link Thetas} we see the first term is equal to
\[ \overline{\int_{\tau \in \mathcal{F}}^{\mathrm{reg}}\left\langle  f(\tau) , iy^{2-2k}\frac{\partial}{\partial z}\overline{(\Theta_{\Delta,r,k}(\tau,z))}\right\rangle \frac{dudv}{v^2}}  {} = {}  \frac{1}{2}\xi_{2-2k}(\Phi_{\Delta,r,k}(z,f)).
\]
We now show the second term vanishes for $z \in \mathbb{H} \backslash Z'_{\Delta,r}(f)$. We know $f^{-}$ decays exponentially, as does the kernel function. We calculate the integral of the $f^{+}$ part by plugging in the expansions given in \eqref{Decomposition} and Definition \ref{Kernel Function} to get
\[ \lim_{t \rightarrow \infty} \sum_{\substack{\lambda \in L+rh \\ Q(\lambda \equiv \Delta Q(h)(\Delta) \\ \lambda \neq 0}} c^{+}\left(- \frac{Q(\lambda)}{|\Delta|}, h \right)\chi_{\Delta}(\lambda)\left( \frac{q_{z}(\lambda)}{y^{2}}\right)^{k} e\left( \frac{-2 Q(\lambda_{z})}{|\Delta|}it\right) t^{1/2}
\] 
recalling $q_{z}(\lambda)=0$ for $\lambda=0$. Using the same analysis as in Theorem \ref{Pointwise} and Theorem \ref{Singularities}  we see the limit vanishes unless $z \in Z'_{\Delta,r}(f)$. 
\end{proof}

We next find the expansion of $\xi_{2-2k}\left(\Phi_{\Delta,r,k} \right)$. By Theorem \ref{Link} this will give the expansion of the Shimura lift which we make explicit in Theorem \ref{big shimura expansion}. We set
\[ 
C_{4} \coloneqq  \tfrac{\overline{\epsilon_{\Delta}} 4\sqrt{2 \pi}\Delta}{i}\left( \tfrac{\pi \Delta}{i \sqrt{2N}}\right)^{k-1}.
 \]

\begin{proposition}\label{xi applied to expansion}
Let $f \in H_{3/2-k,\overline{\rho}}$ with expansion \eqref{Simple Expansion}. Then $\xi_{2-2k}(\Phi_{\Delta,r,k}(z,f))$ 
analytically continues to a holomorphic function on the entire upper-half plane and
\begin{equation}\label{shim expansion equation}
\xi_{2-2k}(\Phi_{\Delta,r,k}(z,f)) = C_{4} \sum_{m \geq 1}\sum_{\substack{d \geq 1 \\ d|m }} \left( \frac{\Delta}{d}\right)\frac{m^{2k-1}}{d^{k}} \overline{c^{-}\left(-\frac{m^{2}}{d^{2}}\frac{|\Delta|}{ 4N},\frac{n}{d}\frac{r}{2N}  \right)}   e\left(m z \right). 
\end{equation}
In the case $k=1,\Delta=1$ we have an additional constant term given by
\[  
\frac{\sqrt{2}}{i k}\left( \frac{1}{ \sqrt{2N}}\right)^{k-1}  \sum_{m \geq 1} m \cdot \overline{c^{+}\left(-\frac{m^{2}}{4N},\frac{r m}{2N} \right)}.
\]
\end{proposition}
\begin{proof}
We deal with each part of our expansion \eqref{big expansion} in turn. We had a constant term \eqref{constant term, expansion} which did not depend on $z$ so this immediately vanishes under $\xi_{2-2k}$. 

We next consider the $c^{+}$ terms (a finite sum over $m$). For $mx+b/\Delta \not\in \mathbb{Z}$ (away from the vertical half-line singularities) and $k \geq 2$ we consider
\[ \left(\frac{\partial}{\partial x} + i \frac{\partial}{\partial y}\right) \left[ B_{k}\left( \left\langle mx + b/D \right\rangle +imy \right) +  \frac{kI_{\mathbb{Z}}(mx+b/D)}{(imy)^{1-k}}\right],
\]
which we see vanishes as $\frac{d}{dx}\left(B_{n}(x)\right) = nB_{n-1}(x)$. In the case of $k=1$, we have 
\[
 \left(\tfrac{\partial}{\partial x} + i \tfrac{\partial}{\partial y}\right) \left[ \mathbb{B}_{1}(mx +b/D) \right] = \frac{m}{2};
 \]
also $\sum_{b (\Delta)}\left( \frac{\Delta}{b}\right) =0 $ unless $\Delta=1$. This gives the constant term.

For the $c^{-}$ terms we use the form given in \eqref{easy form c-}. We see
\begin{align*}  \left(\frac{\partial}{\partial x} + i \frac{\partial}{\partial y}\right)& \left[  e\left(n \left(m z+b/\Delta \right)\right)   + \tfrac{(-1)^{k}\mathrm{sgn}(\Delta)}{(2k-2)!} e\left(-n \left(m z-b/\Delta \right)\right) \Gamma(2k-1,4 \pi n m y)\right]
\\  {} = {} & \frac{ (-1)^{k}\mathrm{sgn}(\Delta)}{i(2k-2)!} e\left(-n \left(m \overline{z}-b/\Delta \right)\right)\left(4 \pi n m \right)^{2k-1}y^{2k-2}.
\end{align*}
Using $\sum_{b(\Delta)}\left( \frac{\Delta}{b} \right)e\left(\frac{nb}{\Delta}\right) = \mathrm{sgn}(\Delta)\left( \frac{\Delta}{n}\right)\sqrt{\Delta}$ and making the substitutions $m \mapsto \frac{m}{n}$ and $n \mapsto d$ we find that applying $\xi_{2-2k,z}$ to the $c^{-}$ terms gives the stated formula. 

Theorem \ref{Locally Harmonic} told us that $\Phi_{\Delta,r,k}(z,f)$ was real analytic for $z \in \mathbb{H} \backslash Z'_{\Delta,r}(f)$. Theorems \ref{The Fourier Expansion} and \ref{Singularities} told us that our Fourier expansion held everywhere (even on the singularities and if $y < \sqrt{-|\Delta|n_{0}/N}$) if we added on appropriate polynomials when crossing walls. These polynomials were of the form 
$q_{z}(\lambda)   =  \frac{-1}{\sqrt{2N}}\left(cN z^{2} -bz + a\right)$
i.e. holomorphic. So they vanish when we apply the  $\xi_{2-2k}$ operator and we can smoothly continue $\xi_{2-2k}\left( \Phi_{\Delta,r,k}(z,f)\right)$ to the entire upper-half plane with the expansion given in \eqref{shim expansion equation}. But the expansion \eqref{shim expansion equation} is clearly holomorphic.
\end{proof}

We set  
\[ 
C_{5} \coloneqq 2i\epsilon_{\Delta}\sqrt{2 N|\Delta|}\left( \tfrac{ \mathrm{sgn}(\Delta)\sqrt{N}}{i \sqrt{2}}\right)^{k-1}. 
\]

\begin{theorem}\label{big shimura expansion}
Let $g(\tau) = \sum_{h \in L'/L} \sum_{n \in \Z+Q(h)} a(n,h)e(n\tau) \mathfrak{e}_h  \in S_{k+1/2, \rho}$. Then 
\[ 
\Phi_{\Delta,r,k}^{*}\left(z,g \right)  = C_{5} \sum_{m \geq 1}\sum_{\substack{d \geq 1 \\ d|m }} \left( \frac{\Delta}{d}\right)d^{k-1} a\left(\frac{m^{2}}{d^{2}}\frac{|\Delta|}{ 4N},\frac{m}{d}\frac{r}{2N}  \right)  e\left(m z \right)
\]
and in the case when $k=1, \Delta=1$ we have an additional constant term given by 
\begin{equation}\label{constant of shimura} \tfrac{1}{i k\sqrt{2}}\left( \tfrac{1}{ \sqrt{2N}}\right)^{k-1}  \sum_{m \geq 1} m \cdot \overline{c^{+}\left(-\tfrac{m^{2}}{4N},\tfrac{r m}{2N} \right)}.
\end{equation}
where $c^{+}(m,h)$ are the coefficients of the principal part of any $f \in H_{3/2-k, \overline{\rho}}$ such that $\xi_{3/2-k}(f)=g$.
\end{theorem}
\begin{proof}
We know there exists an $f \in H_{3/2-k, \overline{\rho}}$ such that $\xi_{3/2-k}(f)=g$. This $f$ must have a Fourier expansion where $c^{-}(n,h) = -\overline{a(-n,h)}(-4 \pi n)^{1/2-k}$. Then as 
\[ \Phi_{\Delta,r,k}^{*}\left(z,g \right) = \frac{1}{2}\xi_{2-2k}\left(\Phi_{\Delta,r,k}\left(z,f \right) \right)\] for $z \in \mathbb{H} \backslash Z_{\Delta,r}(f)$ by Theorem \ref{Link} we then plug in $c^{-}(n,h) = -\overline{a(-n,h)}(-4 \pi n)^{1/2-k}$ into Proposition \ref{xi applied to expansion} to obtain the stated result. 
\end{proof}

\begin{remark}
In the case $k=1,\Delta=1$ we notice the additional constant term depends only on the $c^{+}(n,h)$ coefficients of $f$. This implies that $g$ uniquely determines this term. This is actually an example of \cite[Corollary~1.46]{Hoevel} using the unary theta function, see \cite[Section~4.4]{Hoevel}. This also leads to another corollary \cite[Corollary~3.25]{Hoevel}. Let $f \in M^{!}_{1/2, \overline{\rho}}$ and $\Delta=1$. Then $\xi_{1/2}(f) =g \equiv 0$ and so $\Phi_{\Delta,r,1}^{*}(z,g) \equiv 0$ which means $
\sum_{m \geq 1} m \cdot \overline{c^{+}\left(-\tfrac{m^{2}}{4N},\frac{r m}{2N} \right)} =0$.
\end{remark}

\section{Distributions}
The nature of the singularities found in Theorem \ref{Singularities} leads us to consider these ideas as distributions. We let the space of test functions be smooth forms $g \in A_{\kappa}\left(\Gamma_{0}(N) \right)$ with rapid decay. We denote this space by $A_{\kappa}^{rd}\left(\Gamma_{0}(N) \right)$.

\begin{definition}
For locally harmonic Maass form $h \in LH_{\kappa}\left(\Gamma_{0}(N) \right)$ we define its associated distribution $\left[ h \right]$ on $Y_{0}(N)$ acting on ${A}^{rd}_{\kappa}(\Gamma_0(N))$ by 
\[
 \left[ h \right](g) \coloneqq  \left(g,h \right)_{\kappa} = \int_{Y_{0}(N)} g(z) \overline{h(z)} y^{\kappa} \frac{dx dy}{y^2}.
\]
\end{definition}

It is natural to consider the differential operator $\xi_{\kappa}$ acting on the distribution defined by a locally harmonic Maass form.
\begin{definition}
Let $h \in LH_{\kappa}\left(\Gamma_{0}(N) \right)$ and $\left[ h \right]$ be its associated distribution. Then for $g \in A_{2-\kappa}^{rd}\left(\Gamma_{0}(N) \right)$ we let
\[ \xi_{\kappa}\left[ h \right] \coloneqq  -\left(h,\xi_{2-\kappa}(g) \right)_{\kappa}.
\] 
\end{definition}

Similar to \cite[Section~7]{BF-Duke} we now consider the distributional derivative of the singular theta lift $\left[ \Phi_{\Delta,r,k}(z,f) \right]$.

\begin{theorem}\label{current equation for local}
Let $f \in H_{3/2-k,\overline{\rho}}$. Then for $g \in A_{2k}^{rd}\left(\Gamma_{0}(N) \right)$ we have
\[ \xi_{2-2k}\left[ \Phi_{\Delta,r,k}(z,f) \right](g)  = \left[ \xi_{2-2k}(\Phi_{\Delta,r,k}(z,f)) \right](g) - \sqrt{2|\Delta|} \int_{Z_{\Delta,r}(f)} g(z) q_{z}(\lambda)^{k-1} dz.
\] \end{theorem}
\begin{proof}
Using \cite[Lemma~4.2]{B-Habil} we see $\xi_{2-2k}\left[ \Phi_{\Delta,r,k}(z,f) \right](g)$ is equal to
\begin{align*} 
&  \int_{Y_{0}(N)} \overline{R_{2k-2}\left(y^{2-2k}\overline{\Phi_{\Delta,r,k}(z,f)}\right)} g(z) y^{2k}\frac{dxdy}{y^{2}}  -  \int_{ Y_{0}(N) } d \left(g(z) \Phi_{\Delta,r,k}(z,f) \right) dz .
\end{align*}
The left hand term is equal to $\left[ \xi_{2-2k}(\Phi_{\Delta,r,k}(z,f)) \right](g)$. It remains to consider the right hand term. We decompose $\Phi_{\Delta,r,k}(z,f)$ into its smooth and singular parts, both of which are of weight $2-2k$ for $\Gamma_{0}(N)$, see Theorem \ref{Singularities}. For the smooth part $h(z)$ we know using \cite[Lemma~4.2]{B-Habil}) that 
$ \lim_{t \rightarrow \infty}\int_{ Y_{0}(N) } d \left( g(z) h(z) \right) dz = -\lim_{t \rightarrow \infty} \int_{-1/2}^{1/2}\left[ g(z) h(z) \right]_{y=t} dx$.
This vanishes due to the rapid decay of $g$. We now consider the singular part of $\Phi_{\Delta,r,k}(z,f)$. Using Theorem \ref{Singularities} it suffices to consider
\begin{align}\label{last line}
  &- \sqrt{\frac{|\Delta|}{2}}  \sum_{h \in L'/L}  \sum_{\substack{m \in \mathbb{Z} - \mathrm{sgn}(\Delta)Q(h) \\ m < 0 }} c^{+}\left(m,h \right) \sum_{\lambda \in \Gamma_{0}(N) \backslash L_{-d\Delta,rh}}\chi_{\Delta}(\lambda) \notag\\ & \qquad\times \sum_{\gamma \in \Gamma_{\lambda} \backslash \Gamma_{0}(N)} \int_{ Y_{0}(N) } d \left( g(z)   \frac{(\gamma^{-1}.\lambda,v(z))}{|(\gamma^{-1}.\lambda,v(z))|} q_{z}(\gamma^{-1}.\lambda)^{k-1} \Gamma\left(\frac{1}{2}, \frac{-4 \pi Q((\gamma^{-1}.\lambda)_{z})}{|\Delta|} \right) \right) dz. 
\end{align}
We know $q_{\gamma.z}(\gamma.\lambda) = j(\gamma,z)^{-2}q_{z}(\lambda)$. We also know that $g(z)$ has weight $2k$ and that
$(\gamma.\lambda)_{(\gamma.z)} = \gamma.(\lambda_{z})$ and $(\gamma.\lambda,v(\gamma.z))  = (\lambda,v(z))$. So the integral in \eqref{last line} is equal to
\[ \int_{\Gamma_{\lambda} \backslash \mathbb{H} } d \left( g(z)   \frac{(\lambda,v(z))}{|(\lambda,v(z))|} q_{z}(\lambda)^{k-1} \Gamma\left(\frac{1}{2}, \frac{-4 \pi Q(\lambda_{z})}{|\Delta|} \right) \right) dz.
\]
For $z \in \mathbb{H}$ and any cycle $D_{\lambda}$ we let $\mathrm{dist}(z, D_{\lambda}) \coloneqq \min \left\lbrace |z-w| \mid w \in D_{\lambda} \right\rbrace.$ For any $\epsilon >0$ we let
$ U_{\epsilon}(E) \coloneqq \left\lbrace z \in \mathbb{H} \mid \mathrm{dist}(z, D_{\lambda}) < \epsilon \right\rbrace, $
which defines an $\epsilon$-neighbourhood around the cycle. Using Stokes' theorem we can see that
\[ \lim_{\epsilon \rightarrow 0}\int_{\partial\left(\Gamma_{\lambda} \backslash \left(\mathbb{H}\backslash U_{\epsilon}(\lambda) \right) \right) }  \frac{(\lambda,v(z))}{|(\lambda,v(z))|} \Gamma\left(\frac{1}{2}, \frac{-4 \pi Q(\lambda_{z})}{|\Delta|} \right) dz  =  2\sqrt{\pi} \int_{\Gamma_{\lambda} \backslash D_{\lambda}} 1 dz.
\]
Remember our cycles are oriented so that approaching $(\lambda,v(z))/|(\lambda,v(z))|$ with a left orientation or right orientation gives us $-1$ or $1$ i.e. $2$. The contributions from the $\Gamma_{0}(N)$-equivalent boundary pieces cancel and also as $Q(\lambda_{z}) \rightarrow 0$ the $\Gamma\left(\frac{1}{2}, \frac{-4\pi Q(\lambda_{z})}{|D|} \right)$ term approaches $\Gamma(1/2) = \sqrt{\pi}$. Putting all of this together gives the stated theorem. \end{proof}

We obtain the following refinement of Theorem \ref{Link}.

\begin{theorem}\label{link as distrubtions}
Let $f \in H_{3/2-k,\overline{\rho}}$. Then for $g \in A_{2k}^{rd}\left(\Gamma_{0}(N) \right)$ we have
\[ \xi_{2-2k}\left[ \Phi_{\Delta,r,k}(z,f) \right](g)  + \sqrt{2|\Delta|} \int_{Z_{\Delta,r}(f)} g(z) q_{z}(\lambda)^{k-1} dz = 2\left[ \Phi^{*}_{\Delta,r,k}(z,\xi_{3/2-k}(f)) \right](g).
\]
\end{theorem}

We now consider the case where $g \in S_{2k}\left(\Gamma_{0}(N) \right) \subset A_{2k}^{rd}\left(\Gamma_{0}(N) \right)$. In particular, it vanishes under the $\xi_{2k}$-operator, and we obtain

\begin{corollary}\label{period integral}
Let $f \in H_{3/2-k,\overline{\rho}}$. Then for $g \in S_{2k}\left(\Gamma_{0}(N) \right)$ we have
\[
 \left[ \Phi^{*}_{\Delta,r,k}(z,\xi_{3/2-k}(f)) \right](g) = \sqrt{\frac{|\Delta|}{2}} \int_{Z_{\Delta,r}(f)} g(z) q_{z}(\lambda)^{k-1} dz. 
\]
\end{corollary}

We close with an interpretation of Corollary \ref{period integral} in terms of the Shintani lift.

\begin{definition}
For $g \in S_{2k}(\Gamma_{0}(N))$ the Shintani lift is given by 
\[
\varphi^{*}_{\Delta,r,k}(\tau,g) \coloneqq (g(z),\overline{\Theta^{*}_{\Delta,r,k}(\tau,z)})_{2k}.
\]
\end{definition}

It is clear that $\varphi^{*}_{\Delta,r,k}(\tau,g)$ has weight $k+1/2$ and via Lemma~\ref{Link Thetas} we see that $\varphi^{*}_{\Delta,r,k}(\tau,g)$ is holomorphic with rapid decay since ${\Theta^{*}_{\Delta,r,k}(\tau,z)}$ is rapidly decaying. That is, 
\[
\varphi^{*}_{\Delta,r,k}(\tau,g)  \in S_{k+1/2.\rho}.
\]
Then Corollary~\ref{period integral}  gives us immediately 
\begin{lemma}
Let $f \in H_{3/2-k,\overline{\rho}}$ and $g \in S_{2k}\left(\Gamma_{0}(N) \right)$. Then
\[ \left(\varphi^{*}_{\Delta,r,k}(\tau,g),\xi_{3/2-k}(f)\right)_{k+1/2,\rho} = \sqrt{\frac{|\Delta|}{2}} \int_{Z_{\Delta,r}(f)} g(z) q_{z}(\lambda)^{k-1} dz. 
\]
\end{lemma}
\begin{proof}
This is clear from Corollary~\ref{period integral} after noticing the left hand side is equal to
\begin{align*}  & \int_{Y_{0}(N)} g(z) \int_{\tau \in \mathcal{F}} \left\langle \Theta_{\Delta,r,k}^{*}(\tau,z), \xi_{3/2-k}(f(\tau))   \right\rangle  v^{k+1/2}\frac{dudv}{v^2} y^{2k} \frac{dx dy}{y^2}
\\&  \quad {} = {}   \int_{\tau \in \mathcal{F}}\left\langle   \varphi^{*}_{\Delta,r,k}(\tau,g), \xi_{3/2-k}(f(\tau))  \right\rangle v^{k+1/2}\frac{dudv}{v^2}. \qedhere
\end{align*}
\end{proof}

We then obtain the Fourier expansion of the Shintani lift in terms of cycle integrals by considering $\xi_{3/2-k}(f)$ to be the standard (vector-valued) holomorphic Poincar\'e series of weight $k+1/2$.

\end{document}